\documentclass[11pt]{amsart}
\usepackage{latexsym}
\usepackage{amssymb}
\usepackage{amsmath,amstext,amscd}
\usepackage{mathrsfs}
\usepackage{verbatim}
\usepackage[all]{xy}
\usepackage{amsfonts}
\usepackage{indentfirst}
\usepackage{enumerate}
\usepackage{txfonts}
\usepackage{pxfonts}
\usepackage{bezier,longtable}
\usepackage[labelformat=empty]{caption}
\usepackage{graphicx}

\usepackage{color}
\usepackage{amsthm}

\newtheorem{theo}{Theorem}
\newtheorem{prop}[theo]{Proposition}

\newtheorem{lemma}[theo]{Lemma}

\newtheorem{corollary}[theo]{Corollary}

\newtheoremstyle{remark}
	{\topsep}
	{\topsep}
	{\upshape}
	{}
	{\bfseries}
	{.}
	{5pt plus 1pt minus 1pt}
	{}
\theoremstyle{remark}
\newtheorem{remark}[theo]{Remark}

\renewcommand{\ss}{\mathfrak{s}}
\newcommand{\cc}{\mathfrak{c}}
\renewcommand{\tt}{\mathfrak{t}}
\newcommand{\kk}{\mathfrak{k}}
\newcommand{\id}{\mathrm{id}}
\newcommand{\nn}{\mathfrak{n}}

\newcommand{\hh}{\mathfrak{h}}

\newcommand{\pp}{\mathfrak{p}}
\renewcommand{\gg}{\mathfrak{g}}
\newcommand{\bb}{\mathfrak{b}}

\newcommand{\FF}{\mathcal{F}}

\renewcommand {\phi} {\varphi}

\DeclareMathOperator{\im}{im}

\DeclareMathOperator{\Res}{Res}

\DeclareMathOperator{\Soc}{Soc}
\newcommand{\Ker}{\mathrm{ker}}
\newcommand{\pro}{\mathrm{pro}}

\DeclareMathOperator{\ind}{ind}
\DeclareMathOperator{\Hom}{Hom}

\renewcommand{\Top}{\mathrm{Top}}
\DeclareMathOperator{\TopM}{Top}
\DeclareMathOperator{\Ext}{Ext}

\def\cplus{\hbox{$\subset${\raise0.3ex\hbox{\kern -0.55em ${\scriptscriptstyle +}$}}\ }}

\def\clplus{\hbox{$\subset${\raise0.3ex\hbox{\kern -0.55em ${\scriptscriptstyle +}$}}\ }}
\def\crplus{\hbox{$\supset${\raise1.15pt\hbox{\kern -0.55em ${\scriptscriptstyle +}$}}\ }}

\title{On  Categories of Admissible $\big(\mathfrak{g},\mathrm{sl}(2)\big)$-Modules}

\author[Ivan Penkov]{\;Ivan Penkov}

\address{
Ivan Penkov
\newline Jacobs University Bremen
\newline Campus Ring 1
\newline 28759 Bremen, Germany}
\email{i.penkov@jacobs-university.de}

\author[Vera Serganova]{\;Vera Serganova}

\address{
Vera Serganova
\newline Department of Mathematics
\newline University of California Berkeley
\newline Berkeley CA 94720, USA}
\email{serganov@math.berkeley.edu}

\author[Gregg Zuckerman]{\;Gregg Zuckerman}

\address{
Gregg Zuckerman
\newline Department of Mathematics
\newline Yale University
\newline New Haven, CT 06520, USA}
\email{gregg.zuckerman@yale.edu}

\begin{document}
\maketitle
\begin{abstract} 

	Let $\mathfrak{g}$ be a complex finite-dimensional semisimple Lie algebra and $\mathfrak{k}$ be any $\mathrm{sl}(2)$-subalgebra of $\mathfrak{g}$.  In this paper we prove an earlier conjecture by Penkov and Zuckerman claiming that the first derived Zuckerman functor provides an equivalence between a truncation of a thick parabolic category $\mathcal{O}$ for $\mathfrak{g}$ and a truncation of the category of admissible $(\mathfrak{g}, \mathfrak{k})-$modules. This latter truncated category consists of admissible $(\mathfrak{g}, \mathfrak{k})-$modules with sufficiently large minimal $\mathfrak{k}$-type. We construct an explicit functor inverse to the Zuckerman functor in this setting. As a corollary we obtain an estimate for the global injective dimension of the inductive completion of the truncated category of 
admissible $(\mathfrak{g}, \mathfrak{k})-$modules.

 

Mathematics Subject Classificiation 2010: Primary 17B10, 17B55
\end{abstract}
\section {Introduction}
\label{sec:intro}

Let $\mathfrak{g}$ be a complex finite-dimensional semisimple Lie algebra and $\mathfrak{k} \subseteq \mathfrak{g}$ be a reductive in $\mathfrak{g}$ subalgebra.  An \emph{admissible $(\mathfrak{g},\mathfrak{k})$-module} is a $\mathfrak{g}$-module on which $\mathfrak{k}$ acts semisimply, locally finitely, and with finite multiplicities.  
The study of the category of admissible $(\gg,\kk)$-modules is a main objective of the theory of generalized Harish-Chandra modules, see~\cite{PZ1}.

In the case of a general reductive in $\mathfrak{g}$ subalgebra $\mathfrak{k}$, a central result of the existing theory of generalized Harish-Chandra modules 
is the classification of simple admissible $(\mathfrak{g},\mathfrak{k})$-modules with generic minimal $\mathfrak{k}$-type~\cite{PZ1}.  Other notable results for a general $\mathfrak{k}$ are established in~\cite{PSZ},~\cite{PS1},~\cite{PZ2}, and~\cite{PZ6}.

There are three special cases for $\mathfrak{k}$ in which more detailed information on admissible $(\mathfrak{g},\mathfrak{k})$-modules is available.  First of all, this is the case when $\mathfrak{k}$ is a symmetric subalgebra of $\mathfrak{g}$, i.e., $\mathfrak{k}$ coincides with the fixed points of an involution on $\mathfrak{g}$.  This case, the theory of Harish-Chandra modules, is in the origin of the studies of generalized Harish-Chandra modules.  There is an extensive literature on Harish-Chandra modules, see for instance~\cite{V},~\cite{KV}, and references therein.  (In particular, some remarks on the history of Harish-Chandra modules can be found in~\cite{KV}.)  Another case which has drawn considerable attention is the case when $\mathfrak{k}=\mathfrak{h}$, see for instance~\cite{BL},~\cite{BBL},~\cite{F},~\cite{Fe},~\cite{M},~\cite{GS1},~\cite{GS2}, and references therein.  In both these cases, a classification of simple admissible $(\mathfrak{g},\mathfrak{k})$-modules is available and there has been progress in the study of the category of admissible $(\mathfrak{g},\mathfrak{k})$-modules.

A third natural choice for $\mathfrak{k}$ is to let $\mathfrak{k}$ be isomorphic to $\mathrm{sl}(2)$.  This case ``interpolates'' between the above two cases and is a natural experimentation ground when aiming at the case of a general $\mathfrak{k}$.  For $\mathfrak{k}\simeq \mathrm{sl}(2)$, there is no classification of simple admissible $(\mathfrak{g},\mathfrak{k})$-modules for a general $\mathfrak{g}$ and an arbitrary $\mathrm{sl}(2)$-subalgebra $\mathfrak{k}\subset\mathfrak{g}$; however, for $\kk\simeq\mathrm{sl}(2)$ the partial classification of~\cite{PZ1} can be carried out under much less severe restrictions on the minimal $\mathfrak{k}$-type: the details are explained in~\cite{PZ5} and~\cite{PZ6}.  Since the $\mathfrak{k}$-types are parametrized here simply by nonnegative integers, one can talk about a truncated category of admissible $(\mathfrak{g},\mathfrak{k})$-modules: it consists of finite-length admissible modules whose minimal $\mathfrak{k}$-type is larger or equal a bound $\Lambda$ depending on the pair $(\mathfrak{g},\mathfrak{k})$.  The simple objects of this truncated category have been classified in~\cite{PZ5} (see also~\cite{PZ6}).

The purpose of this paper is to describe the above truncated category of admissible $(\mathfrak{g},\mathfrak{k})$-modules for $\mathfrak{k}=\mathrm{sl}(2)$ by proving that it is equivalent to an explicit full subcategory of a thick parabolic category $\mathcal{O}$ for $\mathfrak{g}$.  In fact, the objects of the truncated category of $(\mathfrak{g},\mathfrak{k})$-modules are constructed by simply applying the Zuckerman (first derived) functor $\Gamma^1$ to a subcategory of a thick parabolic category $\mathcal{O}$.  It was conjectured in~\cite{PZ5} that the functor $\Gamma^1$ yields an equivalence of these categories, and here we prove this conjecture.  We construct a left adjoint to $\Gamma^1$ defined on all finitely generated admissible $(\mathfrak{g},\mathfrak{k})$-modules, and 
then show that, when restricted to the truncated category of admissible $(\gg,\kk)$-modules, this functor is an inverse to the appropriately restricted functor $\Gamma^1$.

The history of $\big(\mathfrak{g},\mathrm{sl}(2)\big)$-modules goes back to the 1940's: a classical example here is the Lorentz pair $\big(\mathrm{sl}(2)\oplus\mathrm{sl}(2), \text{diagonal } \mathrm{sl}(2)\big)$ studied by Harish-Chandra~\cite{HC}, Gelfand-Minlos-Shapiro~\cite{GMS}, and others.  Explaining how exactly the theorem proved in this paper fits in the 70-year history of the topic is a task so complex that we do not really attempt to tackle it.  Nonetheless, we would like to mention that in this subject many equivalences of categories have been established; some relate algebraic categories of $\mathfrak{g}$-modules to geometric categories of sheaves, others relate algebraic categories of $\mathfrak{g}$-modules to other algebraic categories of $\mathfrak{g}$-modules.  The equivalence we establish is clearly of the second kind and could be seen as an analogue of Bernstein-Gelfand's equivalence of a certain subcategory of Harish-Chandra bimodules (or $(\mathfrak{g}\oplus\mathfrak{g},\text{diagonal }\mathfrak{g})$-modules) with category $\mathcal{O}$.  An extension of the geometric techniques introduced by Beilinson and Bernstein from the theory of Harish-Chandra modules to generalized Harish-Chandra modules is not straightforward (see some results in this direction in~\cite{PSZ},~\cite{PS1}, and~\cite{Pe}), and fitting the main result of the present paper into a geometric contex is an open problem.  We show, however, that the algebraic methods from the 1970's (where, in addition to the third author's contribution, we would like to mention to important contributions by Enright-Varadarajan and Enright), together with the more recent ideas of~\cite{PZ1},~\cite{PZ5}, and~\cite{PZ6} (which are building up on Vogan's work), are well suited to yield concrete results about the structure of categories of generalized Harish-Chandra modules.

The paper is structured as follows. We state the main result in Section \ref{sec:main}. In particular, we introduce the functor $\mathrm{B}_1$ which will then be 
shown to be inverse to the functor $\Gamma^1$. In Section \ref{sec:prelim} we present some results which deal  mostly with the 
structure of a semi-thick parabolic category $\mathcal O$ we work with. Section \ref{sec:conjugacy} contains the proof of the adjointness of $\Gamma^1$ and
$\mathrm{B}_1$. The proof of the fact that $\Gamma^1$ and $\mathrm{B}_1$ are mutually inverse equivalences of categories is carried out in steps throughout Sections
\ref{sec:bijection},\ref{sec:exactness} and \ref {sec:endofproof}. In Section \ref{sec:ex} we show that for some blocks of the semi-thick parabolic category $\mathcal O$, the truncation condition
is vacuous, which then helps establish a stronger equivalence of categories for certain central characters. Finally, in Section \ref{sec:appl} we provide an 
application of our equivalence of categories by proving an estimate for the global dimension of the truncated category of admissible $(\gg,\kk)$-modules via a correponding estimate for the truncated semi-thick parabolic category $\mathcal O$.



\bigskip

\noindent{\bf Acknowledgements.}  The first and second authors acknowledge  the hospitality of the Mittag-Leffler Institute in Djursholm where a significant part of this work was written up.  The first author thanks the DFG for partial support through Priority Program SPP 1388 and grant PE 980/6-1.  The second author acknowledges partial support from NSF through grant number DMS 1303301, and both the second and third authors acknowledge the hospitality of Jacobs University Bremen.

\section {Notations and Conventions}
\label{sec:conv}

  The ground field is $\mathbb{C}$.  By $\mathfrak{g}$, we will denote a fixed finite-dimensional semisimple Lie algebra.  We fix also an $\mathrm{sl}(2)$-subalgebra $\mathfrak{k}\subseteq\mathfrak{g}$  By $\kk^\perp$ we denote the orthogonal (with respect to the Killing form) complement of $\kk$ in $\gg$.  The classification of all possible subalgebras $\mathfrak{k}$ up to conjugacy is equivalent to describing all nilpotent orbits in $\mathfrak{g}$, and goes back to Malcev and Dynkin (see~\cite{D} and the references therein).  By a \emph{$\mathfrak{k}$-type} we mean a simple finite-dimensional $\mathrm{sl}(2)$-module.  A simple finite-dimensional $\mathrm{sl}(2)$-module with highest weight $\mu\in\mathbb{Z}_{\geq0}$ is denoted by $V_\kk(\mu)$.  By $\Soc M$ (respectively, $\TopM M$), we denote the socle (respectively, the top) of a $\gg$-module $M$ of finite length.  $\Soc M$ is the maximal semisimple submodule of $M$ and $\TopM M$ is the maximal semisimple quotient of $M$. By $[A:B]$ we denote the multiplicity as a subquotient of a simple module $B$ in a module $A$. $\operatorname{Res}_{\mathfrak q}$ stands for the restriction of a module $M$ to a subalgebra $\mathfrak q$, and $M^{\oplus t}$ stands for the direct sum of $t$ copies of $M$.

\section{Statement of Main Result}
\label{sec:main}

The main result of this paper states that certain categories of $\mathfrak{g}$-modules are equivalent via explicit mutually inverse functors.  Here we define these categories and functors.

Recall that $\mathfrak{g}$ is a finite-dimensional semisimple Lie algebra and $\mathfrak{k}$ is an arbitrary $\mathrm{sl}(2)$-subalgebra of $\mathfrak{g}$.  Fix a standard basis $\big\{e,f,h=[e,f]\big\}$ of $\mathfrak{k}$ and note that $h$ is a semisimple element of $\mathfrak{g}$.  Let $\mathfrak{t}=\mathbb{C}h$ be the toral subalgebra of $\mathfrak{g}$ spanned by $h$.  Define the parabolic subalgebra $\mathfrak{p}$ of $\mathfrak{g}$ by setting 
\begin{equation*}
	\mathfrak{p}:=C(\mathfrak{t})\,\crplus\,\left(\bigoplus_{\alpha(h)>0}\,\gg^\alpha\right)\,.
\end{equation*}
By $\bar{\pp}$ we denote the opposite parabolic subalgebra
\begin{equation*}
	\bar{\mathfrak{p}}=C(\mathfrak{t})\,\crplus\,\left(\bigoplus_{\alpha(h)<0}\,\gg^\alpha\right)\,,
\end{equation*}
where $C(\mathfrak{t})$ is the centralizer of $h$ in $\mathfrak g$. We also set
$$\nn:=\bigoplus_{\alpha(h)>0}\,\gg^\alpha.$$

Let $\mathcal{C}_{\bar{\pp},\tt}$ be the category of finite-length $\gg$-modules which are $\bar{\pp}$-locally finite, $\tt$-semisimple, and $\tt$-integral (i.e., $h$ acts with integer eigenvalues).  Informally, $\mathcal{C}_{\bar{\pp},\tt}$ is a ``semi-thick'' (``thick in all directions except the $\tt$-direction'') parabolic category $\mathcal{O}$.  By $\mathcal{C}_{\bar{\pp},\tt,n}$ for $n\in\mathbb{Z}_{\geq 0}$, we denote the $n$-truncated category $\mathcal{C}_{\bar{\pp},\tt}$, i.e., the full subcategory of $\mathcal{C}_{\bar{\pp},\tt}$ consisting of objects all $\tt$-weights $\mu$ of which satisfy $\mu(h)\geq n$.  We also assign an integer $\Lambda$ to the pair $(\gg,\kk)$: we set $\Lambda=\frac{1}{2}\left(\lambda_1+\lambda_2\right)$, where $\lambda_1$ (respectively, $\lambda_2$) is the maximum (resp., submaximum) weight of $\tt$ in $\gg/\kk$.  Here and below, we identify $\tt$-weights with integers via the correspondence $\mu \leadsto \mu(h)$.

Denote by $\mathcal{C}_\kk$ the category of admissible $(\gg,\kk)$-modules of finite length, i.e., the category of $\gg$-modules $M$ of finite length on which $\kk$ acts locally finitely and such that $\dim\Hom_\kk(L,M)<\infty$ for any $\kk$-type $L$.  By $\mathcal{C}_{\kk,n}$ for $n\in\mathbb{Z}_{\geq 0}$, we denote the full subcategory of $\mathcal{C}_\kk$ consisting of $\gg$-modules $M$ such that $\Hom_\kk(L,M)\neq 0$ implies $\dim L>n$.

We now describe two functors: $\Gamma_{\kk,\tt}$ and $\mathrm{B}^{\kk,\tt}$.  $\Gamma_{\kk,\tt}$ is the functor of $\kk$-finite vectors in a $(\gg,\tt)$-module.  That is, if $M$ is a $(\gg,\tt)$-module then
\begin{equation*}
	\Gamma_{\kk,\tt}M:=\Big\{m\in M\,\big|\,\dim U(\kk)\cdot m<\infty\Big\}\,,
\end{equation*}
and $\Gamma_{\kk,\tt}M$ is a $\gg$-submodule of $M$.
It is well known (and easy to see) that $\Gamma_{\kk,\tt}$ is a left-exact functor.  In what follows we set $\Gamma:=\Gamma_{\kk,\tt}$ and denote the right derived functors $R^i\Gamma_{\kk,\tt}$ by $\Gamma^i$.  ($\Gamma^i$ is known as the \emph{$i$-th Zuckerman functor}.)  By definition, $\Gamma^i$ is a functor from $(\gg,\tt)\text{\rm-mod}$ to $(\gg,\kk)\text{\rm-mod}$.  It is proved in~\cite{PZ5} that the restriction of $\Gamma^i$ to $\mathcal{C}_{\bar\pp,\tt,n+2}$ is a well-defined functor from 
$\mathcal{C}_{\bar\pp,\tt,n+2}$ to $\mathcal{C}_{\kk,n}$.  We denote this functor also by $\Gamma^i$.

We now define a functor
\begin{equation*}
	\mathrm{B}^{\kk,\tt}:(\mathfrak{g},\mathfrak{t})^\text{\rm fg}\text{\rm-mod}\rightsquigarrow\mathcal{C}_{\bar{\pp},\tt,n+2}\,,
\end{equation*}
where $(\mathfrak{g},\mathfrak{t})^\text{\rm fg}\text{\rm-mod}$ stands for the category of finitely generated $(\gg,\tt)$-modules.  

Throughout the rest of the paper, $\theta:Z_{U(\gg)}\to\mathbb{C}$ will denote a fixed central character.  If $M$ is a $\gg$-module, then $M^\theta$ stands for the vectors in $M$ on which $z-\theta(z)$ acts locally nilpotently for any $z\in Z_{U(\gg)}$.  By $\ell$ we denote a variable positive integer.  

Let $\mathcal{C}^{\theta,\ell}_{\bar{\pp},\tt,n+2}$ be the subcategory of $\mathcal{C}_{\bar{\pp},\tt,n+2}$ consisting of modules $M$ with $M=M^\theta$ and such that $\hh$ acts via Jordan blocks of size at most $\ell$.  We note that $\mathcal{C}^{\theta,\ell}_{\bar{\pp},\tt,n+2}$ is a finite-length category which has an injective cogenerator $I^{\theta,\ell}_{n+2}$.  This fact is proved in Lemma~\ref{lem:lemma1} below.  We set
\begin{equation*}
	\left(\mathrm{B}^{\kk,\tt}\right)^{\theta,\ell} X:= X\,\Big/\,\left(\bigcap_{\varphi \in \Hom_\gg\left(X,\,I^{\theta,\ell}_{n+2}\right)} \ker\varphi\right)
\end{equation*}
for $X\in (\gg,\tt)^\text{\rm fg}\text{\rm-mod}$.  Lemma~\ref{lem:lemma3} below claims that $\left(\mathrm{B}^{\kk,\tt}\right)^{\theta,\ell} X\in \mathcal{C}^{\theta,\ell}_{\bar{\pp},\tt,n+2}$, which shows that $\left(\mathrm{B}^{\kk,\tt}\right)^{\theta,\ell}$ is the ``largest quotient'' of $X$ lying in $\mathcal{C}^{\theta,\ell}_{\bar{\pp},\tt,n+2}$.  

Next, we notice that there is a canonical surjective homomorphism
\begin{equation*}
	\left(\mathrm{B}^{\kk,\tt}\right)^{\theta,\ell}  X \twoheadrightarrow\left(\mathrm{B}^{\kk,\tt}\right)^{\theta,\ell-1} X\,,
\end{equation*}
i.e., that $\left\{\left(\mathrm{B}^{\kk,\tt}\right)^{\theta,\ell}  X\right\}$ is an inverse system of $\bar{\pp}$-locally finite $(\gg,\tt)$-modules.  We set
\begin{equation*}
	\left(\mathrm{B}^{\kk,\tt}\right)^{\theta} X:=\lim_{\longleftarrow}\,\left(\mathrm{B}^{\kk,\tt}\right)^{\theta,\ell}  X\,.
\end{equation*}
It is easy to see that $\left(\mathrm{B}^{\kk,\tt}\right)^\theta$ is a right-exact functor from $(\gg,\tt)^\text{\rm fg}\text{\rm-mod}$ to $\gg\text{\rm-mod}$, and we denote by $\left(\mathrm{B}^{\kk,\tt}\right)^\theta_j$ its left derived functors, that is $\left(\mathrm{B}^{\kk,\tt}\right)^\theta_jX=L_j\left(\mathrm{B}^{\kk,\tt}\right)^\theta X$ for $X\in (\gg,\tt)^\text{\rm fg}\text{\rm-mod}$.

Let $\mathcal{C}^{\theta}_{\kk,n}$ and $\mathcal{C}^\theta_{\bar\pp,\tt,n+2}$ be the respective subcategories of $\mathcal{C}_{\kk,n}$ and $\mathcal{C}_{\bar\pp,\tt,n+2}$ consisting of $\gg$-modules $M$ with $M=M^\theta$.  Corollary~\ref{cor:cor1_of_prop2} below states that in fact $\left(\mathrm{B}^{\kk,\tt}\right)^\theta_j$ is a well-defined functor from $\mathcal{C}_{\kk,n}^\theta$ to $\mathcal{C}_{\bar{\pp},\tt,n+2}^\theta$.  As $\kk$ and $\tt$ are fixed, in what follows we set $\mathrm{B}^{\theta,\ell}:=\left(\mathrm{B}^{\kk,\tt}\right)^{\theta,\ell}$, $\mathrm{B}^\theta:=\left(\mathrm{B}^{\kk,\tt}\right)^\theta$, and $\mathrm{B}_j^\theta:=\left(\mathrm{B}^{\kk,\tt}\right)^\theta_j$.  By the same letters we also denote the restrictions of these functors to the category $\mathcal{C}^\theta_{\kk,n}$.  

The main result of this paper is the following.

\begin{theo}\label{main}
For $X\in\mathcal{C}_{\kk,n}$, let  $\mathrm{B}_j X:=\displaystyle\bigoplus_\theta\,\mathrm{B}_j^\theta X$.  Then, for any $n\geq \Lambda$, the functors
\begin{equation*}
	\Gamma^1:\mathcal{C}_{\bar{\pp},\tt,n+2}\rightsquigarrow\mathcal{C}_{\kk,n}
\end{equation*}
and
\begin{equation*}
	\mathrm{B}_1:\mathcal{C}_{\kk,n}\rightsquigarrow\mathcal{C}_{\bar{\pp},\tt,n+2}
\end{equation*}
are mutually inverse equivalences of categories.
\end{theo}

\begin{remark}
	Let us note that for the Lorentz pair $\big(\mathrm{sl}(2)\oplus\mathrm{sl}(2),\mathrm{diagonal}\,\mathrm{sl}(2)\big)$ a description of the category $\mathcal{C}_\kk$ was given by I. Gelfand and V. Ponomarev in \cite{GP} already in 1967.
\end{remark}

\section{Preparatory Results}
\label{sec:prelim}

We start with two general results.
	
\begin{prop}\label{prop:spec} 

\begin{enumerate}[a)]
	\item For any $(\gg,\kk)$-module $X$ there exists a
singly graded spectral sequence converging to $H_i(\nn,X)$ such that its $E^1$-term has the form 
\begin{equation}\label{spectral}
E_i^1=H_0(\nn_\kk,X)\otimes\Lambda^i(\nn\cap\kk^\perp)\oplus H_1(\nn_\kk,X)\otimes\Lambda^{i-1}(\nn\cap\kk^\perp).
\end{equation}
	\item If $X\in\mathcal{C}_{\kk,n}$ for $n\geq 0$, then the $\nn$-homology $H_\bullet(\nn,X)$ is finite dimensional.

\end{enumerate}
\end{prop}

\begin{proof}

\begin{enumerate}[a)]
\item	The statement follows from Proposition 1 in \cite{PZ1} if one sets $i$ to be the total degree $a+b$ in the notation of \cite{PZ1}.  The statement holds more generally for any $\gg$-module $X$ but the assumption that $X$ is a $(\gg,\kk)$-module is sufficient for us.

\item This follows from the more general statement of Proposition 3.5 in~\cite{PZ5}.
\end{enumerate}
\end{proof}

Let $M=\oplus_{p\in\mathbb C} M_p$ be an admissible $(\gg,\tt)$-module where $M_p$ is the $\tt$-weight space in $M$ of weight $p$: by definition,
$hm=pm$ for $m\in M_p$. Set $M^*_\tt:=\oplus_{p\in\mathbb C} M^*_p$. Then $M^*_\tt$ is a well-defined admissible $(\gg,\tt)$-module. 
Similarly let $X=\oplus_{\mu\in\mathbb Z_{\geq 0}}\tilde V_\kk(\mu)$ be an admissible $(\gg,\kk)$-module. Here $\tilde V_\kk(\mu)$ stands for the $V_\kk(\mu)$-isotypic component in $X$. Then
$X^*_\kk:=\oplus_{\mu\in\mathbb Z_{\geq 0}}\tilde V_\kk(\mu)^*$ is a well-defined admissible $(\gg,\kk)$-module. 
Moreover, $(\bullet)_\tt^*$ and $(\bullet)_\kk^*$ are well-defined contravariant functors (in fact antiequivalences) on the respective categories
of admissible $(\gg,\tt)$-modules and $(\gg,\kk)$-modules.

In what follows we will use the composition of the functors $(\bullet)_\tt^*$ and $(\bullet)_\kk^*$ with the twist by the
Cartan involution of $\gg$. The so obtained new functors are denoted respectively by  $(\cdot)_\tt^\vee$ and $(\cdot)_\kk^\vee$. These functors 
preserve the respective $\tt$- and $\kk$-characters of the modules.

A duality theorem proved in \cite{EW} implies
\begin{prop}\label{prop:duality} For any admissible $(\gg,\tt)$-module $M$ there is a natural isomorphism of admissible $(\gg,\kk)$-modules
$$(\Gamma^iM)^\vee_\kk\simeq \Gamma^{2-i}(M^\vee_\tt).$$
\end{prop}

In what follows $E$ stands for a finite-dimensional simple $C(\tt)$-module on which $h$ acts via a natural number $|E|$.  Often, we consider $E$ as a $\bar{\pp}$-module by setting $\bar{\nn}\cdot E:=0$.  In this case, we set also 
	\begin{equation*}
		M(E):=U(\gg)\underset{U(\bar{\pp})}{\otimes}E
	\end{equation*}
	and we let $L(E)$ be the unique simple quotient of $M(E)$.  Then $L(E)^\vee\simeq L(E)$, and $M(E)^\vee$ is an indecomposable object of $\mathcal{C}_{\bar{\pp},\tt}$ with $\Soc\,M(E)^\vee=L(E)^\vee\simeq L(E)$.

The following proposition is a summary of preliminary results concerning the specific categories we study in this paper.

\begin{prop}\label{prop:Zuck} Let $n\geq\Lambda$. Then
\begin{enumerate}[a)]
\item $\Gamma^1:\mathcal{C}_{\bar{\pp},\tt,n+2}\to\mathcal{C}_{\kk,n}$ is an exact functor which maps a 
simple object to a simple object, and induces a bijection on the isomorphism classes of simple objects in $\mathcal{C}_{\bar{\pp},\tt,n+2}$ and in $\mathcal{C}_{\kk,n}$, respectively.    Moreover, the simple $(\gg,\kk)$-module  $\Gamma^1L(E)$ has minimal $\kk$-type $|E|-2$.

\item $\Gamma^1L(E)$ and $\Top\,\Gamma^1M(E)$ are isomorphic simple $(\gg,\kk)$-modules with minimal $\kk$-type $|E|-2$, and the isotypic components of the minimal $\kk$-types of $\Gamma^1M(E)$ and $\Gamma^1L(E)$ are isomorphic.
\end{enumerate}
\end{prop}

\begin{proof}
Part a) follows directly from the results of~\cite{PZ5}, see Corollary 6.4 and Section 9.  Part b) is a consequence of the above mentioned results and the fact that the functor $\Gamma^1$ commutes with $(\bullet)^\vee$ according to Proposition~\ref{prop:duality}.
\end{proof}

\begin{remark}
	Since $\Gamma^1$ preserves central character, Proposition~\ref{prop:Zuck}, a) implies in particular the existence of a bijection between the isomorphism classes of simple objects in $\mathcal{C}^\theta_{\bar\pp,\tt,n+2}$ and $\mathcal{C}^\theta_{\kk,n}$ for $n\geq \Lambda$.  Without the condition $n\geq \Lambda$, no such bijection exists in general.  For instance, if $(\gg,\kk)$ is the Lorentz pair and $\theta$ is the central character of a finite-dimensional $\gg$-module of the form $V\boxtimes V$ for a simple finite-dimensional $\mathrm{sl}(2)$-module $V$, then $\mathcal{C}^\theta_{\bar\pp,\tt,2}$ has $3$ pairwise nonisomorphic simple objects while $\mathcal{C}^\theta_{\kk,0}$ has two nonisomorphic simple objects. 
\end{remark}

The rest of the section is devoted to results on $\bar{\pp}$-locally finite modules.

Note that $C(\tt)$ is a reductive subalgebra: we denote by $\ss$ the derived subalgebra of $C(\tt)$, and by $\cc$ the center of $C(\tt)$.
We choose a Borel subalgebra $\bb$ of $\gg$ such that $e\in\bb\subset\pp$.  Let $\hh$ be a Cartan subalgebra of $\bb$ containing $h$. Then $\cc\subset \hh$.
Let $\FF_{C(\tt),\tt}$ be the category of locally finite $C(\tt)$-modules semisimple over $\tt$ with integral $h$-weights. 
By  $\FF^\ell_{C(\tt),\tt}$ we denote the subcategory of $\FF_{C(\tt),\tt}$ consisiting of modules on which $\cc$ acts via Jordan blocks of size less than or equal to $\ell$.
Clearly
$$\FF_{C(\tt),\tt}=\lim_{\longrightarrow}\FF^\ell_{C(\tt),\tt}.$$
Note that $({\bullet})^\vee$ is also a well-defined functor on the category $\FF^\ell_{C(\tt),\tt}$ (but not on $\mathcal{F}_{C(\tt),\tt}$).

\begin{lemma}\label{small} Let $S$ be a simple finite-dimensional $\ss$-module and $\lambda\in\cc^*$ be a $\tt$-integral weight. 
Define $E$ as $S\otimes \mathbb{C}_\lambda$ where $\mathbb{C}_\lambda$ is a one-dimensional $\cc$ module with weight $\lambda$. Let $I_\lambda^{\ell}$ denote the ideal in $S(\cc)$ generated by
$h-\lambda(h)$ and $(z-\lambda(z))^\ell$ for all $z\in\cc$. Then 
\begin{enumerate}[a)]
\item Every simple object in $\FF_{C(\tt),\tt}$ is isomorphic to $E$ for some choice of $S$ and $\lambda$. Furthermore, $E^\vee\simeq E$. 
\item $E^\ell:=E\otimes (S(\cc)/I_\lambda^\ell)$ is a projective cover of $E$ and $\left(E^\ell\right)^\vee$ is an injective hull of $E$ in $\FF^\ell_{C(\tt),\tt}$.
\item $\bar E:=\displaystyle \lim_{\longrightarrow} \left(E^\ell\right)^\vee$ is an injective hull of $E$ in $\FF_{C(\tt),\tt}$.
\end{enumerate}
\label{lem:lem_inj_hull}
\end{lemma}
\begin{proof} (a) is obvious. To show (b), note that $S$ is projective in the category of locally finite $\ss$-modules, and that $E^\ell$ is the maximal quotient of
the induced module $U(C(\tt))\otimes_{U(\ss)}S$ lying in in $\FF^\ell_{C(\tt),\tt}$. Then (c) is clearly a corollary of (b).
\end{proof}


	The following generalizes basic results in \cite{BGG}.
	
	\begin{lemma}
		For any $\ell>0$ and any $n\geq 0$, the abelian category $\mathcal{C}^{\theta,\ell}_{\bar{\pp},\tt,n+2}$ has a unique, 
up to isomorphism, minimal injective cogenerator $I^{\theta,\ell}_{n+2}$.  Moreover, the $\tt$-weight spaces of $I^{\theta,\ell}_{n+2}$ are finite dimensional.
		\label{lem:lemma1}
	\end{lemma}

	\begin{proof} 
We denote by $\FF^\ell_{\pp,\tt}$ the category of $\pp$-modules whose restrictions to $C(\tt)$ belong to $\FF^\ell_{C(\tt),\tt}$, and by
$\FF^\ell_{\pp,\tt,n+2}$ the subcategory consisting of modules whose
$\tt$-weights are bounded from below by $n+2$. Let $E$ be as in Lemma~\ref{lem:lem_inj_hull}.  Endow $E$ with a $\pp$-module structure by letting $\nn$ act trivially on $E$, and consider the $\pp$-module
\begin{equation}
\Gamma_\tt \pro^{\pp}_{C(\tt)}\left(\left(E^{\ell}\right)^\vee\right)=\Gamma_\tt\Hom_{U(C(\tt))}\left(U(\pp),\left(E^{\ell}\right)^\vee\right)\,.
\label{eq:gammapromodule}
\end{equation}
Since $\pro^{\pp}_{C(\tt)}(\cdot)$ preserves injectivity and the functor of $\tt$-weight vectors $\Gamma_\tt$ is right adjoint to the inclusion of the category of $\pp$-modules semisimple
over $\tt$ into the category of all $\pp$-modules, and hence also preserves injectivity, the $\pp$-module (\ref{eq:gammapromodule})
is an injective hull of $E$ in $\FF^\ell_{\pp,\tt}$. Consequently, the truncated submodule $\left(\Gamma_\tt\pro^{\pp}_{C(\tt)}\left(\left(E^{\ell}\right)^\vee\right)\right)_{\geq n+2}$ of (\ref{eq:gammapromodule}), spanned by all 
$\tt$-weight spaces with
weights greater or equal than $n+2$, is an injective hull of $E$ in $\FF^\ell_{\pp,\tt,n+2}$. 

Set
$$J^\ell(E):=\left[\Gamma_\tt\pro^{\gg}_{\pp}\left(\left(\Gamma_\tt\pro^{\pp}_{C(\tt)}\left(\left(E^{\ell}\right)^\vee\right)\right)_{\geq n+2}\right)\right]^\theta.$$
Then, by a similar argument, $J^\ell(E)$ is injective in $\mathcal{C}^{\theta,\ell}_{\bar{\pp},\tt,n+2}$ and we have an embedding of $\gg$-modules 
$L(E)\hookrightarrow J^\ell(E)$ induced by the embedding
of $\pp$-modules $E\hookrightarrow \left(\pro^{\pp}_{C(\tt)}\left(\left(E^{\ell}\right)^\vee\right)\right)_{\geq n+2}$. Moreover, it is easy to check that
$J^\ell(E)$ has finite-dimensional $\tt$-weight spaces.

Note that, up to isomorphisms, $\mathcal{C}^{\theta,\ell}_{\bar{\pp},\tt,n+2}$ has finitely many simple objects $L(E_1),\dots, L(E_r)$. Each of them has a unique, up to isomorphism, injective hull
$I^\ell\left(E_j\right)$ which is a submodule of $J^\ell(E_j)$. Then $I^{\theta,\ell}_{n+2}$ is the direct sum $\bigoplus_{j=1}^r\,I^\ell\left(E_j\right)$.
	\end{proof}
	
	\begin{corollary}\label{cor:cor2-0}
		Let $A^{\theta,\ell}_{n+2}:=\mathrm{End}_\gg\, I^{\theta,\ell}_{n+2}$.  Then $\mathcal{C}^{\theta,\ell}_{\bar\pp,\tt,n+2}$ is equivalent to the category of finite-dimensional $A^{\theta,\ell}_{n+2}$-modules.
	\end{corollary}
	
	Let $\mathcal{C}^{\theta,\text{\rm ind}}_{\bar{\pp},\tt,n+2}$ be the category of inductive limits of objects from $\mathcal{C}_{\bar{\pp},\tt,n+2}^{\theta,\ell}$.

        \begin{corollary} \label{cor:cor2} For any $n$, the category $\mathcal{C}^{\theta,\text{\rm ind}}_{\bar{\pp},\tt,n+2}$ has a unique, up to isomorphism, 
minimal injective cogenerator $I^\theta_{n+2}$.  Moreover, $I^\theta_{n+2}=\displaystyle\lim\limits_{\longrightarrow}\,I^{\theta,\ell}_{n+2}$. In particular, the 
category $\mathcal{C}^{\theta,\text{\rm ind}}_{\bar{\pp},\tt,n+2}$ has 
enough injectives.
          \end{corollary}
          
          In fact, if $I$ is any injective object in $\mathcal{C}^{\theta,\text{\rm ind}}_{\bar{\pp},\tt,n+2}$, then $I$ is a direct limit, $\lim\limits_{\longrightarrow}\,I^\ell$, for injective objects $I^\ell\in\mathcal{C}^{\theta,\ell}_{\bar{\pp},\tt,n+2}$.

Recall the definition of $\bar E$ from Lemma \ref{small}, and let
$$\overline{W}(E):=\Gamma_\tt\pro^{\gg}_{\pp}(\bar{E})$$

\begin{lemma}\label{injectivecoVerma} Let $\FF_{\bar{\pp},\tt}$ be the category of locally finite $\bar{\pp}$-modules such that their restrictions to $C(\tt)$ lie in $\FF_{C(\tt),\tt}$.  Then $\overline{W}(E)$ is an injective hull of $E$ in $\FF_{\bar{\pp},\tt}$.  Moreover, $\Res_{C(\tt)}\overline{W}(E)\simeq \bigoplus_{\alpha}\,\bar{F}_\alpha$ for some finite-dimensional irreducible $C(\tt)$-modules $F_\alpha$.
\end{lemma}
\begin{proof} By the Poincare--Birkhoff--Witt Theorem we have an isomorphism 
$$\operatorname{Res}_{\bar\pp}\overline{W}(E)\simeq \Gamma_\tt\Hom_{C(\tt)}(U(\bar\pp),\bar E)=\Gamma_\tt\pro^{\bar\pp}_{C(\tt)}\bar E.$$
Since $\bar{E}$ is an injective module in $\mathcal{F}_{C(\tt),\tt}$, $\Gamma_{\tt}\pro^{\bar\pp}_{C(\tt)}\bar{E}$ is an injective module in $\mathcal{F}_{\bar\pp,\tt}$.  Hence $\Res_{\bar\pp}\overline{W}(E)$ is an injective module in $\mathcal{F}_{\bar\pp,\tt}$.  

Let $S(E)$ be the socle of $\Res_{\bar\pp}\overline{W}(E)$ as a module over $C(\tt)$.  Since $\Res_{C(\tt)}\overline{W}(E)$ is locally $C(\tt)$-finite, it is an essential extension of $S(E)$.  We conclude that $\Res_{C(\tt)}\overline{W}(E)$ is an injective hull of $S(E)\simeq\bigoplus_{\alpha}\,F_\alpha$ for some finite-dimensional irreducible $C(\tt)$-modules $F_\alpha$.  

We now show that $\Res_{C(\tt)}\overline{W}(E)\simeq\bigoplus_\alpha\,\bar{F}_\alpha$.  Let $T(E)$ be the direct sum of the $C(\tt)$-modules $\bar{F}_\alpha$.  Then $T(E)$ is an essential extension of $S(E)$.  Because
$U\big(C(\tt)\big)$ is left Noetherian, $T(E)$ is injective.  Consequently, $T(E)$ is an injective hull of its socle.  But $\Soc T(E)=\bigoplus_{\alpha}\,F_\alpha=S(E)$.  Since $\Res_{C(\tt)}\overline{W}(E)$ is also an injective hull of $S(E)$, there is an isomorphism $\Res_{C(\tt)}\overline{W}(E)\simeq\bigoplus_{\alpha}\,\bar{F}_\alpha$.

\end{proof}

We say that an object $M\in \mathcal{C}^{\text{\rm ind}}_{\bar{\pp},\tt}$ admits a {\it co-Verma filtration} if there exists a finite filtration
$$0=M_0\subset M_1\subset\dots\subset M_t=M$$  whose successive quotients $M_{i+1}/M_i$ are isomorphic to $\overline{W}(E_i)$ for simple $C(\tt)$-modules
$E_1,\dots, E_t$.

\begin{lemma}\label{coVermafilt} Let $M$ be an object of $\mathcal{C}^{\text{\rm ind}}_{\bar{\pp},\tt}$. Then $M$ admits a co-Verma filtration if and only if
$\operatorname{Res}_{\bar\pp}M$ is injective in $\FF_{\bar{\pp},\tt}$ with socle of finite length.
\end{lemma}
\begin{proof} If $M$ admits a co-Verma filtration, then $\operatorname{Res}_{\bar\pp}M$ is a direct sum of modules of the form $\overline{W}(F)$, and by Lemma \ref{injectivecoVerma}
$\operatorname{Res}_{\bar\pp}M$ is injective in $\FF_{\bar{\pp},\tt}$ with $\bar\pp$-socle of finite length.

To prove the opposite assertion, choose a simple $\bar\pp$-submodule $E\subset \Res_{\bar\pp}M$ with minimal $|E|$.  The existence of $E$ follows from the fact that the socle of $\Res_{\bar\pp}M$ has finite length. Let $k$ be the multiplicity of $E$ in $\Soc\Res_{\bar\pp}M$.
Then we have a surjective
morphism $\varphi:\Res_{\pp}M\to \bar E^{\oplus k}$ of $\pp$-modules ($\varphi|_{\bar{E}^{\otimes k}}$ being the identity map) which induces a morphism $\tilde\varphi: M\to \overline{W}(E)^{\oplus k}$ of $\gg$-modules by Frobenius reciprocity. 
Since $\operatorname{Res}_{\bar\pp}M$ is injective in $\mathcal{F}_{\bar\pp,\tt}$, $\operatorname{Res}_{\bar\pp}M$ is an injective hull of its socle, i.e., 
$$\operatorname{Res}_{\bar\pp}M\simeq\Res_{\bar\pp}\left(\overline{W}(E)^{\oplus k}\oplus \bigoplus_{|F|>|E|}\overline{W}(F)\right)$$
 by Lemma~\ref{injectivecoVerma}.  Moreover, $\Hom_{\bar\pp}\left(\Res_{\bar\pp}\overline{W}(F),\Res_{\bar\pp}\overline{W}(E)\right)=0$ if $|F|>|E|$.
Therefore $\tilde\varphi\Biggl(\Res_{\bar\pp}\left(\bigoplus_{|F|>|E|}\overline{W}(F)\right)\Biggr)=0,$ and $$\big.\tilde\varphi\big|_{\Res_{\bar\pp}\overline{W}(E)^{\oplus k}}:\Res_{\bar\pp}\overline{W}(E)^{\oplus k}\to \Res_{\bar\pp}\overline{W}(E)^{\oplus k}$$ is an isomorphism of $\bar\pp$-modules
since it is induced by the identity map $\varphi|_{\bar{E}^{\oplus k}}:\bar E^{\oplus k}\to \bar E^{\oplus k}$. 

Set $Q:=\Ker\,\tilde{\varphi}$. Then $\operatorname{Res}_{\bar\pp}Q$ is isomorphic to
$\Res_{\bar\pp}\left(\bigoplus_{|F|>|E|}\overline{W}(F)\right)$, and hence $Q$ satisfies all conditions of the lemma. So we can finish the proof by induction on the length of the socle of $\Res_{\bar\pp}M$. 
\end{proof}

\begin{corollary}\label{technical} Let $R=M\oplus N$ for some  $M,N\in\mathcal{C}^{\theta,\text{\rm ind}}_{\bar{\pp},\tt}$.
Suppose that $R$ admits a co-Verma filtration. Then $M$ and $N$ also admit  co-Verma 
filtrations.
\end{corollary}
\begin{proof} A direct summand of an injective module is injective, so the statement follows from Lemma~\ref{coVermafilt}.
\end{proof}
	
	\begin{lemma} \label{lem:lemma2} Let $I(E)$ be an injective hull of $L(E)$ in $\mathcal{C}^{\theta,\text{\rm ind}}_{\bar{\pp},\tt,n+2}$.  
Then $I(E)/\overline{W}(E)$ admits a co-Verma filtration with successive quotients isomorphic to $\overline{W}(D)$ for $|D|<|E|$.
	\end{lemma}
        \begin{proof} Let
$J(E):=\left[\Gamma_\tt\pro^{\gg}_{\pp}\left(\left(\Gamma_\tt\pro^{\pp}_{C(\tt)}\left(\bar{E}\right)\right)_{\geq n+2}\right)\right]^\theta$.
The $\pp$-module $\Gamma_\tt\left(\left(\pro^{\pp}_{C(\tt)}(\bar E)\right)_{\geq n+2}\right)$ has a finite filtration with successive quotients $\bar{D}$ with $|D|\leq |E|$. Moreover,
the quotient $\Biggr(\Gamma_\tt\left(\left(\pro^{\pp}_{C(\tt)}(\bar E)\right)_{\geq n+2}\right)\Biggl)/\bar{E}$ has a filtration with successive quotients $\bar{D}$ for $|D|<|E|$. Therefore
$J(E)/\overline{W}(E)$ admits a co-Verma filtration with successive quotients $\overline{W}(D)$ with  $|D|<|E|$. 

Similarly as in the proof of Lemma \ref{lem:lemma1}, $I(E)$ is a direct summand of $J(E)$. Therefore, by Lemma \ref{technical}, 
$I(E)$ has a filtration as desired.
\end{proof}

\begin{corollary}
$I^\theta_{n+2}$ admits a co-Verma filtration.
\end{corollary}
\section{Adjointness of $\mathrm{B}_1$ and $\Gamma^1$}
\label{sec:conjugacy}	

In this section, $n$ is an arbitrary nonnegative integer.

	\begin{lemma}
		For any $\ell\in\mathbb{Z}_{>0}$, $\mathrm{B}^{\theta,\ell}$ is a right-exact functor from $(\gg,\tt)^\text{\rm fg}\text{\rm-mod}$ into $\mathcal{C}^{\theta,\ell}_{\bar{\pp},\tt,n+2}$ (in particular, $\mathrm{B}^{\theta,\ell} X$ has finite length for $X\in (\gg,\tt)^\text{\rm fg}\text{\rm-mod}$).
		\label{lem:lemma3}
	\end{lemma}
	
	\begin{proof}
		Fix $X\in(\gg,\tt)^{\text{fg}}\text{-mod}$.  Then $\Hom_\gg\left(X,I^{\theta,\ell}_{n+2}\right)$ is finite dimensional.  This follows from the fact that the $\tt$-weight spaces of $I^{\theta,\ell}$ are finite dimensional.  As a consequence, $\mathrm{B}^{\theta,\ell}X$ is isomorphic to a submodule of a finite direct sum of copies of $I^{\theta,\ell}_{n+2}$.  Since $I^{\theta,\ell}_{n+2}$ has finite length, $\mathrm{B}^{\theta,\ell} X$ also has finite length, and is an object of $\mathcal{C}^{\theta,\ell}_{\bar\pp,\tt,n+2}$.
		
		The fact that $\mathrm{B}^{\theta,\ell}$ is right-exact follows from the observation that $\mathrm{B}^{\theta,\ell}$ is left adjoint to the inclusion functor $\mathcal{C}^{\theta,\ell}_{\bar{\pp},\tt,n+2} \rightsquigarrow (\gg,\tt)^\text{\rm fg}\text{\rm-mod}$, i.e., 
		\begin{equation*}
			\Hom_{\mathcal{C}^{\theta,\ell}_{\bar{\pp},\tt,n+2}}\left(\mathrm{B}^{\theta,\ell} X,M\right)\simeq\Hom_{\gg,\tt}(X,M)
		\end{equation*}
		for any $X\in (\gg,\tt)^\text{\rm fg}\text{\rm-mod}$ and $M\in\mathcal{C}^{\theta,\ell}_{\bar{\pp},\tt,n+2}$.  Indeed, a left adjoint to a left-exact functor is right-exact.
	\end{proof}
	
	\begin{lemma}
		For any $X\in(\gg,\tt)^\text{\rm fg}\text{\rm-mod}$ and $j \in \mathbb{Z}_{\geq 0}$, we have
		\begin{equation*}
			\mathrm{B}_j^\theta X\simeq\lim_{\longleftarrow}\,\mathrm{B}_j^{\theta,\ell} X\,,
		\end{equation*}
		where $\mathrm{B}_j^{\theta,\ell}$ is the $j$-th left derived functor of $\mathrm{B}^{\theta,\ell}$.
		\label{lem:inverselimitbeta}
	\end{lemma}
	
	\begin{proof}
		Let $P_\bullet$ be a projective resolution of $X$ in the category $(\gg,\tt)^\text{fg}\text{-mod}$.  By definition, $\mathrm{B}_\bullet^\theta X$ is the homology of the complex $\mathrm{B}^\theta P_\bullet$, and $\mathrm{B}_\bullet^{\theta,\ell}X$ is the homology of the complex $\mathrm{B}^{\theta,\ell}P_\bullet$.  Moreover,
		\begin{equation*}
			\mathrm{B}^\theta P_\bullet=\lim_{\longleftarrow}\,\mathrm{B}^{\theta,\ell}P_\bullet\,.
		\end{equation*}
	By Lemma~\ref{lem:lemma3}, for every $\ell$, $j$ and $q$, $\mathrm{B}_j^{\theta,\ell}X_q$ has finite length as a $\gg$-module.  So, Lemma~\ref{lem:inverselimitbeta} follows from the Mittag-Leffler Principle.
	\end{proof}
	
	If $\left\{A^\ell\right\}$ is an inverse system of objects from $\mathcal{C}^{\theta,\ell}_{\bar{\pp},\tt}$ for variable $\ell\to\infty$, we set
	\begin{equation*}
		\Hom_{\gg}^{\text{\rm cont}}\left(\lim_{\longleftarrow}\,A^\ell,M\right) := \lim_{\longrightarrow}\,\Hom_\gg \left(A^\ell,M\right)
	\end{equation*}
	for any $\gg$-module $M$.
	
	\begin{prop}
		Let $I$ be injective in $\mathcal{C}^{\theta,\text{\rm ind}}_{\bar\pp,\tt,n+2}$.  Then, for any finitely generated $(\gg,\tt)$-module $X$ and for any $j \in \mathbb{Z}_{\geq 0}$, 
		\begin{equation*}
			\Hom_\gg^{\text{\rm cont}}\left(\mathrm{B}^\theta_j X,I\right)\simeq\Ext_{\gg,\tt}^j\left(X,I\right)\,.
		\end{equation*}
		\label{prop:prop1}
	\end{prop}
	
	\begin{proof}
		Let $P_\bullet$ be as in the proof of Lemma~\ref{lem:inverselimitbeta}.  Then
		\begin{equation*}
			\Hom_\gg^{\text{\rm cont}}\left(\mathrm{B}^\theta_j X,I\right)=\Hom_{\gg}^{\text{\rm cont}}\left(H_j(\mathrm{B}^\theta P_\bullet),I\right)=\Hom^{\text{cont}}_{\gg}\left(H_j\left(\lim_{\longleftarrow}\,\mathrm{B}^{\theta,\ell}P_\bullet\right),I\right)\,.
		\end{equation*}
		Since
		\begin{equation*}
			H_j\left(\lim_{\longleftarrow}\,\mathrm{B}^{\theta,\ell} P_\bullet\right)\simeq\lim_{\longleftarrow}\,H_j\left(\mathrm{B}^{\theta,\ell} P_\bullet\right)
		\end{equation*}
		by the Mittag-Leffler Principle, we have
		\begin{equation*}
			\Hom^{\text{cont}}_{\gg}\left(H_j\left(\lim_{\longleftarrow}\,\mathrm{B}^{\theta,\ell}P_\bullet\right),I\right)\simeq\Hom_\gg^{\text{\rm cont}}\left(\lim_{\longleftarrow}\,H_j\left(\mathrm{B}^{\theta,\ell} P_\bullet\right),I\right)=\lim_{\longrightarrow}\,\Hom_\gg\Big(H_j\left(\mathrm{B}^{\theta,\ell} P_\bullet\right),I\Big)\,.
		\end{equation*}
		Next, the injectivity of $I$ in $\mathcal{C}^{\theta,\text{\rm ind}}_{\bar\pp,\tt,n+2}$ implies
		\begin{equation*}
			\Hom_\gg\Big(H_j\left(\mathrm{B}^{\theta,\ell}  P_\bullet\right),I\Big)\simeq H_j\Big(\Hom_\gg\left(\mathrm{B}^{\theta,\ell}  P_\bullet,I\right)\Big)\,.
		\end{equation*}
		Consequently,
		\begin{equation*}
			\lim_{\longrightarrow}\,\Hom_\gg\Big(H_j\left(\mathrm{B}^{\theta,\ell}  P_\bullet\right),I\Big)\simeq \lim_{\longrightarrow}\,H_j\Big(\Hom_\gg\left(\mathrm{B}^{\theta,\ell}  P_\bullet,I\right)\Big)\,,
		\end{equation*}
		and since homology commutes with direct limits, 
		\begin{equation}
			\lim_{\longrightarrow}\,H_j\Big(\Hom_\gg\left(\mathrm{B}^{\theta,\ell}  P_\bullet,I\right)\Big)\simeq H_j\Big(\lim_{\longrightarrow}\,\Hom_\gg\left(\mathrm{B}^{\theta,\ell}  P_\bullet,I\right)\Big)\,.
			\label{eq:eqhomdirectlimits1}
		\end{equation}
		
		Recalling that $I=\displaystyle\lim_{\longrightarrow} I^\ell$, we notice that
                \begin{equation}
\Hom_\gg\left(\mathrm{B}^{\theta,\ell}  P_\bullet,I\right)=\Hom_\gg\left(\mathrm{B}^{\theta,\ell}  P_\bullet,I^\ell\right)=\Hom_\gg\left(P_\bullet,I^\ell\right)\,.\label{eq:eqhomdirectlimits2}
                  \end{equation}
Furthermore, since $P_\bullet$ is finitely generated,
               \begin{equation}
			\lim_{\longrightarrow}\,\Hom_\gg\left(P_\bullet,I^\ell\right) \simeq \Hom_\gg\left(P_\bullet,I\right)\,.\label{eq:eqhomdirectlimits3}
                  \end{equation}
Therefore, (\ref{eq:eqhomdirectlimits1}), (\ref{eq:eqhomdirectlimits2}), and (\ref{eq:eqhomdirectlimits3}) yield
		\begin{equation*}
			\lim_{\longrightarrow}\,H_j\Big(\Hom_\gg\left(\mathrm{B}^{\theta,\ell}  P_\bullet,I\right)\Big) \simeq  H_j\Big(\Hom_\gg\left(P_\bullet,I\right)\Big)\,.
		\end{equation*}
		Since $H_j\Big(\Hom_\gg\left(P_\bullet,I\right)\Big)=\Ext_{\gg,\tt}^j\left(X,I\right)$, we obtain
		\begin{equation*}
			\Hom_{\gg}^{\text{\rm cont}}\left(\mathrm{B}^\theta_j X,I\right)\simeq \Ext_{\gg,\tt}^j\left(X,I\right)
		\end{equation*}
		as desired.
	\end{proof}
	
	Recall that $I^\theta_{n+2}$ is an injective cogenerator of the category $\mathcal{C}^{\theta,\text{\rm ind}}_{\bar{\pp},\tt,n+2}$.
	\begin{prop}
		For any $X\in\mathcal{C}_{\kk,n}$ and any $j \geq 0$, $\Ext_{\gg,\tt}^j\left(X,I^\theta_{n+2}\right)$ is finite dimensional.
		\label{prop:prop2}
	\end{prop}
	
	\begin{proof}
		By Lemma~\ref{lem:lemma2}, it suffices to show that $\dim\Ext_{\gg,\tt}^j\left(X,\overline{W}(E)\right)<\infty$ for any $E$ with $\overline{W}(E) \in \mathcal{C}^{\theta,\text{\rm ind}}_{\bar\pp,\tt,n+2}$.  By the Shapiro Lemma, 
		\begin{equation*}
			\Ext_{\gg,\tt}^j\left(X,\overline{W}(E)\right)=\Ext^i_{\gg,\tt}\Big(X,\Gamma_t\mathrm{pro}_{\pp}^\gg\bar{E}\Big) \simeq \Ext_{\pp,\tt}^j\left(X,\bar{E}\right)\,.
		\end{equation*}
		Since by the injectivity of $\bar{E}$ as a $C(\tt)$-module we have $\Ext_{\pp,\tt}^j\left(X,\bar{E}\right)\simeq \Hom_{C(\tt)}\Big(H_j(\nn,X),\bar{E}\Big)$, we conclude that
		\begin{equation}
			\Ext_{\gg,\tt}^i\left(X,\overline{W}(E)\right)\simeq \Hom_{C(\tt)}\Big(H_j(\nn,X),\bar{E}\Big)\,.
			\label{eq:eq23}
		\end{equation}		
		Now the statement follows from the finite-dimensionality of $H_j(\nn,X)$, see Proposition~\ref{prop:spec}, b).
	\end{proof}
	
	\begin{corollary}
		For $X\in\mathcal{C}_{\kk,n}$ and any $j\geq 0$, we have $\mathrm{B}_j X\in\mathcal{C}_{\bar\pp,t,n+2}$.
		\label{cor:cor1_of_prop2}
	\end{corollary}
	
	\begin{proof}
		By Lemma~\ref{lem:lemma3}, $\Hom_\gg\left(\mathrm{B}_j^{\theta,\ell}X,I^\theta_{n+2}\right)$ is finite dimensional for any $\ell$.  By Propositions~\ref{prop:prop1} and~\ref{prop:prop2}, $\Hom_{\gg}^{\text{\rm cont}}\left(\mathrm{B}_j^\theta X,I^\theta_{n+2}\right)=\lim\limits_{\longrightarrow}\,\Hom_\gg\left(\mathrm{B}_j^{\theta,\ell}X,I^\theta_{n+2}\right)$ is finite dimensional.  By the definition of the direct limit functor, we have, for sufficiently large $s$,
		\begin{equation*}
			\Hom_\gg\left(\mathrm{B}_j^{\theta,s}\,X,I^\theta_{n+2}\right)\simeq \Hom_{\gg}\left(\mathrm{B}_j^{\theta,s+1}X,I^\theta_{n+2}\right)
		\end{equation*}
		under the $\gg$-module map $\alpha^s:\mathrm{B}_j^{\theta,s+1}X\to\mathrm{B}_j^{\theta,s}X$.		
		Since $I^\theta_{n+2}$ is an injective cogenerator for the category $\mathcal{C}_{\bar\pp,\tt,n+2}^{\theta,\ind}$, we conclude that the map $\alpha^s$ is an isomorphism for sufficiently large $s$.
		
		By Lemma~\ref{lem:inverselimitbeta} and the definition of the inverse limit functor, we conclude that $\mathrm{B}_j^\theta X\simeq\mathrm{B}_j^{\theta,s}X$ for large enough $s$.  Since $\mathrm{B}^{\theta'}_jX\neq 0$ for only finitely many $\theta'$, $\mathrm{B}_jX$ has finite length, or equivalently $\mathrm{B}_jX\in\mathcal{C}_{\bar\pp,\tt,n+2}$.
	\end{proof}
	
	\begin{corollary}
		For any $X\in\mathcal{C}_{\kk,n}$ and any injective object $I\in\mathcal{C}^{\theta,\text{\rm ind}}_{\bar\pp,\tt,n+2}$, 
		\begin{equation*}
			\Hom_{\gg}^{\text{\rm cont}}\left(\mathrm{B}^\theta_j X,I\right)=\Hom_\gg\left(\mathrm{B}^\theta_j X,I\right)\,.
		\end{equation*}
		\label{cor:cor2_of_prop2}
	\end{corollary}
	
	\begin{proof}
		This follows from the isomorphism $\mathrm{B}_j^\theta X\simeq \mathrm{B}_j^{\theta,s}X$ for large enough $s$.
	\end{proof}
	
	\begin{prop}
		For any $X\in\mathcal{C}_{\kk,n}$ and $M\in\mathcal{C}_{\bar\pp,\tt,n+2}$, 
		\begin{equation}
			\Hom_\gg\left(\mathrm{B}_1 X,M\right)\simeq \Hom_\gg\left(X,\Gamma^1 M\right)\,,
			\label{eq:prop3}
		\end{equation}
		so $\mathrm{B}_1:\mathcal{C}_{\kk,n}\to\mathcal{C}_{\bar\pp,\tt,n+2}$ and $\Gamma^1:\mathcal{C}_{\bar\pp,\tt,n+2}\to\mathcal{C}_{\kk,n}$ are adjoint functors.
		\label{prop:prop3}
	\end{prop}
	
	\begin{proof}
		It clearly suffices to prove (\ref{eq:prop3}) for $X\in\mathcal{C}^\theta_{\kk,n}$ and $M\in \mathcal{C}^\theta_{\bar\pp,\tt,n+2}$.  By Proposition 7.9
in \cite{PZ5}, 
		\begin{equation}
			\Hom_\gg\left(X,\Gamma^1I\right) \simeq \Ext_{\gg,\tt}^1(X,I)
			\label{eq:prop3_1}
		\end{equation}
		for any injective object $I \in \mathcal{C}^{\theta,\text{\rm ind}}_{\bar\pp,\tt,n+2}$.
		By Proposition~\ref{prop:prop1} and Corollary~\ref{cor:cor2_of_prop2}, 
		\begin{equation}
			\Ext_{\gg,\tt}^1(X,I)\simeq \Hom_\gg\left(\mathrm{B}_1X,I\right)\,.
			\label{eq:prop3_2}
		\end{equation}
		
		Consider a part of an injective resolution of $M$ in $\mathcal{C}^{\theta,\text{\rm ind}}_{\bar\pp,\tt}$, $0\to M\to I_0 \to I_1$.  Since $\Gamma^1$ is an exact functor (Proposition~\ref{prop:Zuck}), the following sequence is also exact: $0\to\Gamma^1M \to \Gamma^1 I_0 \to \Gamma^1 I_1$.  Next, applying $\Hom_\gg(X,\bullet)$, we obtain an exact sequence
		\begin{equation*}
			0\to \Hom_\gg\left(X,\Gamma^1M\right) \to \Hom_\gg\left(X,\Gamma^1I_0\right)\to \Hom_\gg\left(X,\Gamma^1 I_1\right)\,.
		\end{equation*}
		By (\ref{eq:prop3_1}) and (\ref{eq:prop3_2}), we have a diagram
		\begin{equation*}
			\includegraphics[width=12cm]{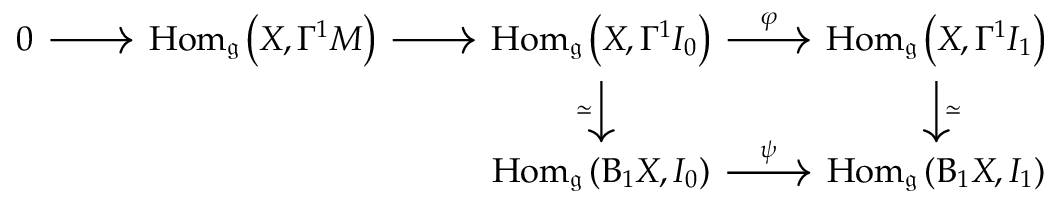}
		\end{equation*}
		which is commutative as the identifications
		\begin{equation*}
			\Hom_\gg\left(X,\Gamma^1I\right)\simeq \Ext^1_{\gg,\tt}(X,I)= \Hom_\gg\left(\mathrm{B}_1 X,I\right)
		\end{equation*}
		are functorial in $I$.  Since $\Hom_\gg\left(\mathrm{B}_1X,\bullet\right)$ is left-exact, we conclude that
		\begin{equation*}
			\ker\psi \simeq \Hom_\gg\left(\mathrm{B}_1 X,M\right)\,.
		\end{equation*}
		Finally, $\ker\varphi\simeq\ker\psi$, and we are done.
	\end{proof}
	\begin{corollary}
	      $\mathrm{B}_1:\mathcal{C}_{\kk,n}\to\mathcal{C}_{\bar\pp,\tt,n+2}$ is a right-exact functor.
	      \label{corr:corr11}
	\end{corollary}
	\begin{proof}
	 This is a direct consequence of the fact that $\mathrm{B}_1$ is a left adjoint functor.
	\end{proof}

\section{$\mathrm{B}_1$ is a bijection on isomorphism classes of simple modules in $\mathcal{C}_{\kk,n}$ and $\mathcal{C}_{\bar\pp,\tt,n+2}$.}
\label{sec:bijection}

	As stated in Proposition~\ref{prop:Zuck}, for $n\geq\Lambda$ the functor $\Gamma^1$ induces a bijection between the sets of isomorphism classes of simple objects in the categories $\mathcal{C}_{\bar\pp,\tt,n+2}$ and $\mathcal{C}_{\kk,n}$.  The main result of this section is:
	
	\begin{prop}
		For $n\geq \Lambda$, the functor $\mathrm{B}_1$ induces a bijection of sets of isomorphism classes of simple objects of $\mathcal{C}_{\kk,n}$ and $\mathcal{C}_{\bar\pp,\tt,n+2}$, inverse to the bijection induced by $\Gamma^1$.
		\label{prop:prop11}
	\end{prop}
	
	In the rest of the paper we assume that $n\geq \Lambda$.  We set $X(E):=\Gamma^1L(E)$ for $L(E)\in\mathcal{C}^\theta_{\bar\pp,\tt,n+2}$. Then $X(E)$ is a simple object of $\mathcal{C}_{\kk,n}$, and all simple objects in $\mathcal{C}_{\kk,n}$ are of this form for appropriate simple $C(\tt)$-modules.
	
	\begin{lemma}
		Let $X\in\mathcal{C}^\theta_{\kk,n}$ have the property that the isotypic component of the minimal $\kk$-type of $X$ is isomorphic to the isotypic component of the minimal $\kk$-type of $X(E)\in\mathcal{C}^\theta_{\kk,n}$.  Then
		\begin{equation*}
			\left[\mathrm{B}_1^\theta X:L(F)\right] = \left\{
				\begin{array}{ll}
					0	&	\text{for }|F|<|E|\,,
					\\
					\leq 1	&	\text{for }|F|=|E|\,.
				\end{array}
			\right.
		\end{equation*}
		\label{lem:lem12}
	\end{lemma}
	
	\begin{proof}
		Observe first that the $\tt$-weights of $H_0\left(\nn_\kk,X\right)$ are less than or equal to $-\big(|E|-2\big)$.  Therefore the $\tt$-weights of $H_0\left(\nn_\kk,X\right)\otimes \left(\nn_\kk\cap \kk^\perp\right)$ are less than or equal to $-|E|+2+\lambda_1$.  This shows that $\Big(H_0\left(\nn_\kk,X\right)\otimes\left(\nn_\kk\cap\kk^\perp\right)\Big)_p=0$ for $n+2\leq p <|E|$.   Indeed, the inequalities $n+2\leq p<|E|$ and $p\leq -|E|+2+\lambda_1$ yield $|E|\leq \frac{\lambda_1}{2}-\frac{\lambda_2}{2}\leq \frac{\lambda_1}{2}$, which contradicts our assumption that $|E|\geq 2+\frac{\lambda_1+\lambda_2}{2}\geq 2+\frac{\lambda_1}{2}$.
		
		Next, note that Kostant's Theorem applied to $\kk$ gives
		\begin{equation*}
			H_1\left(\nn_\kk,X\right)_p=\left\{
				\begin{array}{ll}
					0	&	\text{for }p<|E|\,,
					\\
					\mathbb{C}	&	\text{for }p=|E|\,.
				\end{array}
			\right.
		\end{equation*}
		Therefore the spectral sequence (\ref{spectral}) of Proposition \ref{prop:spec} implies
\begin{equation*}
(E^1_1)_p=\begin{cases} 0\,\,\text{for}\,\,n+2\leq p<|E|\\ \mathbb C \,\,\text{for}\,\, p=|E|\end{cases}
,\,(E^1_0)_{|E|}=(E^1_2)_{|E|}=0\,.
\end{equation*}
Hence,
		\begin{equation*}
			H_1\left(\nn,X\right)_p=\left\{
				\begin{array}{ll}
					0	&	\text{for }n+2\leq p<|E|\,,
					\\
					\mathbb{C}	&	\text{for }p=|E|\,.
				\end{array}
			\right.
		\end{equation*}
		Furthermore, for any $D$
		\begin{equation*}
			\dim\Ext^1_{\gg,\tt}\left(X,\overline{W}(D)\right)=\dim\Hom_{C(\tt)}\big(H_1(\nn,X),\bar{D}\big) \leq \dim H_1(\nn,X)_{|D|}
		\end{equation*}
		by (\ref{eq:eq23}).  This yields
		\begin{equation*}
			\dim\Ext^1_{\gg,\tt}\left(X,\overline{W}(D)\right)= \left\{
				\begin{array}{ll}
					0	&	\text{for }|D|<|E|\,,
					\\
					\leq 1	&	\text{for }|D|=|E|\,.
				\end{array}
			\right.
		\end{equation*}
		Consequently, since the injective hull $I(F)$ of $L(F)$ in $\mathcal{C}^{\theta,\textrm{ind}}_{\bar\pp,\tt,n+2}$ admits a co-Verma filtration with successive quotients isomorphic to $\overline{W}(D)$ for $|D|\leq |F|$, and $\overline{W}(E)$ enters $I(E)$ with multiplicity $1$, we have
		\begin{equation*}
			\dim\Ext^1_{\gg,\tt}\big(X,I(F)\big)= \left\{
				\begin{array}{ll}
					0	&	\text{for }|F|<|E|\,,
					\\
					\leq 1	&	\text{for }|F|=|E|\,.
				\end{array}
			\right.
		\end{equation*}
		Finally, 
		\begin{equation*}
			\dim \Ext_{\gg,\tt}^1\big(X,I(F)\big)=\dim \Hom_\gg \left(\mathrm{B}_1^\theta X,I(F)\right)=\left[\mathrm{B}_1^\theta X:L(F)\right]\,,
		\end{equation*}
		and the lemma is proved.
	\end{proof}
	
	\begin{corollary}
		Set $Y(E):=\Gamma^1M(E)$ under the assumption that $M(E)\in\mathcal{C}^\theta_{\bar\pp,\tt,n+2}$.  Then
		\begin{equation*}
			\left[\mathrm{B}_1Y(E):L(F)\right]=\left[\mathrm{B}_1X(E):L(F)\right]=\left\{
				\begin{array}{ll}
					0	&	\text{for }|F|<|E|\,,
					\\
					1	&	\text{for }F\simeq E\,.
				\end{array}
			\right.
		\end{equation*}
		\label{cor:cor13}
	\end{corollary}
	
	\begin{proof}
		For $|F|<|E|$, the statement follows directly from Lemma~\ref{lem:lem12} as the isotypic components of the minimal $\kk$-types of $Y(E)$ and $X(E)$ are isomorphic by Proposition~\ref{prop:Zuck}, b).  If $F\simeq E$, then 
		\begin{equation*}
		\Hom_\gg\left(\mathrm{B}_1X(E),L(E)\right)\simeq \Hom_\gg\big(X(E),X(E)\big)\,,
		\end{equation*} 
		so the identity homomorphism $X(E)\to X(E)$ provides a nonzero homomorphism $\mathrm{B}_1X(E)\to L(E)$.  Since $\Gamma^1$ is exact and $\mathrm{B}_1$ is right-exact, this homomorphism is in fact a composition of surjections $\mathrm{B}_1Y(E)\to \mathrm{B}_1X(E)\to L(E)$, in particular, $\left[\mathrm{B}_1Y(E):L(E)\right]\geq 1$ and $\left[\mathrm{B}_1X(E):L(E)\right]\geq 1$.  On the other hand, $\left[\mathrm{B}_1Y(E):L(E)\right]\leq 1$ by Lemma~\ref{lem:lem12}; hence,
		\begin{equation*}
			\left[\mathrm{B}_1Y(E):L(E)\right]=\left[\mathrm{B}_1X(E):L(E)\right]=1\,.
		\end{equation*}
	\end{proof}
	
	\begin{corollary}
		$\mathrm{B}_1Y(E)\simeq M(E)$.
		\label{cor:cor14}
	\end{corollary}
	
	\begin{proof}
		By the adjointness of $\mathrm{B}_1$ and $\Gamma^1$, we have a canonical nonzero homomorphism
		\begin{equation*}
			\varphi:\mathrm{B}_1Y(E)\to M(E)
		\end{equation*}
		induced by the identity homomorphism $Y(E)\to Y(E)$.  Note that $\TopM Y(E)$ is isomorphic to $X(E)$ by Proposition \ref{prop:Zuck}, b).  
Next, by the adjointness of $\mathrm{B}_1$ and 
$\Gamma^1$,
		\begin{equation*}
			\Hom_\gg\left(\mathrm{B}_1Y(E),L(F)\right)\simeq \Hom_\gg\big(Y(E),X(F)\big)\,,
		\end{equation*}
		and as $\Hom_\gg\big(Y(E),X(F)\big)\neq 0$ only for $F\simeq E$, we see that
		\begin{equation*}
			\TopM \mathrm{B}_1Y(E)\simeq L(E)\,.
		\end{equation*}
		
		Since $\varphi\neq 0$, $\varphi$ induces an isomorphism
		\begin{equation*}
			\TopM \mathrm{B}_1Y(E)\simeq L(E)=\TopM M(E)\,.
		\end{equation*}
		Consequently, $\varphi$ is surjective.  Let $N=\ker\varphi$. Since the exact sequence
$$0\to N\to \mathrm{B}_1Y(E)\to M(E)\to 0$$
does not split, the assumption $N\neq 0$ would imply $[N:L(F)]\neq 0$ for some $|F|<|E|$, i.e., a contradiction with Corollary~\ref{cor:cor13}. Hence, $N=0$ and $\varphi$ is an 
isomorphism.
	\end{proof}
	
	\begin{corollary}
		$\TopM \Gamma^1\mathrm{B}_1X(E)\simeq X(E)$.
		\label{cor:cor15}
	\end{corollary}
	
	\begin{proof}
		$\mathrm{B}_1$ is right-exact, hence the surjective homomorphism $Y(E)\to X(E)$ yields a surjective homomorphism $M(E)\simeq \mathrm{B}_1Y(E) \to \mathrm{B}_1X(E)$.  By applying $\Gamma^1$ we obtain a surjective homomorphism $Y(E)\to \Gamma^1\mathrm{B}_1X(E)$, hence $\TopM (\Gamma^1\mathrm{B}_1X(E))\simeq \TopM Y(E)\simeq X(E)$.
	\end{proof}
	
	\begin{corollary}
		$\mathrm{B}_1X(E)\simeq L(E)$.
		\label{cor:cor16}
	\end{corollary}
	
	\begin{proof}
		By Corollary~\ref{cor:cor13}, $\mathrm{B}_1X(E)\neq 0$.  We have an isomorphism
		\begin{equation*}
			\Hom_\gg\left(X(E),\Gamma^1\mathrm{B}_1X(E)\right) \simeq \Hom_\gg\left(\mathrm{B}_1X(E),\mathrm{B}_1X(E)\right)\,.
		\end{equation*}
		Hence, we have a nonzero (and therefore injective) homomorphism
		\begin{equation*}
			\alpha:X(E)\to\Gamma^1\mathrm{B}_1X(E)
		\end{equation*}
		corresponding to the identity homomorphism $\mathrm{B}_1X(E)\to\mathrm{B}_1X(E)$.  
		
		Once again, Corollary~\ref{cor:cor13} implies $\left[\mathrm{B}_1X(E):L(E)\right]=1$.  Since $\Gamma^1$ is exact and is a bijection on isomorphism classes of simple modules, we have 
		\begin{equation}
			\left[\Gamma^1\mathrm{B}_1X(E):X(E)\right]=1\,.
			\label{eq:cor16_0}
		\end{equation} 
		By Corollary~\ref{cor:cor15}, there is a surjective homomorphism $$\beta:\Gamma^1\mathrm{B}_1X(E)\to X(E)\,.$$  Equation (\ref{eq:cor16_0}) now implies that $\beta\alpha\neq 0$ and $\alpha\beta\neq 0$.  Thus, $X(E)$ is a direct summand of $\Gamma^1\mathrm{B}_1X(E)$.  But Corollary~\ref{cor:cor15} shows that $\Gamma^1\mathrm{B}_1X(E)$ is indecomposable.  We conclude that
		\begin{equation*}
			\Gamma^1\mathrm{B}_1X(E)\simeq X(E)=\Gamma^1L(E)\,.
		\end{equation*}
		As before, $\Gamma^1$ is exact and is a bijection on isomorphism classes of simple modules.  So, $\mathrm{B}_1X(E)$ is a simple $\gg$-module, and
		$
			\mathrm{B}_1X(E)\simeq L(E)
		$.
	\end{proof}
	
\section{Exactness of $\mathrm B_1$}
\label{sec:exactness}

The goal of this section is to prove the following.

	\begin{prop}\label{prop:exact}
		$\mathrm{B}_1:\mathcal{C}_{\kk,n}\to\mathcal{C}_{\bar\pp,\tt,n+2}$ is an exact functor.
		\label{prop:prop17}
	\end{prop}

Our main effort will go into proving
	
	\begin{lemma}
		$\mathrm{B}_2Y(E)=0$ for $X(E)\in\mathcal{C}_{\kk,n}$.
		\label{lem:lem18}
	\end{lemma}
	
	Note that it suffices to show that $H_2\big(\nn,Y(E)\big)_{|F|}=0$ for all $X(E),X(F)\in\mathcal{C}_{\kk,n}$.  Indeed, 
the implication 
$$\left(H_2\big(\nn,Y(E)\big)_{|F|}=0\,\text{for all} \,X(E),X(F)\in\mathcal{C}_{\kk,n}\right)\,\Rightarrow \left(\mathrm{B}_2Y(E)=0\,\text{for all}\, X(E)\in\mathcal{C}_{\kk,n}\right)$$ follows from the following three facts:
	\begin{enumerate}
		\item $\Ext^2_{\gg,\tt}\big(Y(E),I(F)\big)\simeq \Hom_\gg\left(\mathrm{B}_2Y(E),I(F)\right)$;
		\item $I(F)/\overline{W}(F)$ has a co-Verma filtration with factors isomorphic to $\overline{W}\left(F'\right)$ for $\left|F'\right| < |F|$ (Lemma~\ref{lem:lemma2});
		\item $\dim\Ext^2_{\gg,\tt}\left(Y(E),\overline{W}(F)\right)=\dim\Hom_{C(\tt)}\Big(H_2\big(\nn,Y(E)\big),\bar{F}\Big)\leq \dim H_2\big(\nn,Y(E)\big)_{|F|}$,
see (\ref{eq:eq23}).
	\end{enumerate}
	
	To prove that $H_2\big(\nn,Y(E)\big)_{|F|}=0$ for all $X(E),X(F)\in\mathcal{C}_{\kk,n}$, we give another construction of 
the functor $\Gamma^1: \mathcal{C}_{\bar{\pp},\tt,n+2}\to\mathcal{C}_{\kk,n} $.  
Denote by $U_e(\gg)$ the enveloping algebra $U(\gg)$ localized by the multiplicative set $\left\{e^n\right\}_{n\in\mathbb{Z}_{\geq 1}}$.   The localized algebra $U_e(\kk)$ is a subalgebra of $U_e(\gg)$. For any $\gg$-module (resp., $\kk$-module) $M$, set
$$\mathcal D^\gg_e(M):=U_e(\gg)\otimes_{U(\gg)}M,\,\,\,\,\mathcal D^\kk_e(M):=U_e(\kk)\otimes_{U(\kk)}M\,.$$

\begin{lemma}\label{lem:localization} If $M$ is a $\gg$-module on which $e$ acts injectively, we have an isomorphism of $\kk$-modules
$${\Res}_\kk\mathcal D^\gg_e(M)\simeq \mathcal D^\kk_e(M).$$
  \end{lemma}
\begin{proof} There is an embedding $\psi:M\hookrightarrow  D^\gg_e(M)$. 
By Frobenius reciprocity $\psi$ induces a morphism $\tilde\psi:D^\kk_e(M)\to  D^\gg_e(M)$. Since $U_e(\gg)=U_e(\kk)S(\kk^\perp)$, the morphism
$\tilde\psi$ is surjective. Let us show that $\tilde\psi$ is also injective. Since $e$ acts injectively on $M$, we see that $\mathcal{D}_e^\gg(M)$ is an essential extension of $M$.  Therefore, the fact that $\Ker\,\tilde\psi\cap M=0$ suffices to conclude that $\tilde\psi$ is injective.
\end{proof}

Suppose that a $\gg$-module $M$ is free over $\mathbb C[e]$ and locally finite over $\mathbb{C}[f]$. Then we have an embedding
$$M\hookrightarrow \Gamma_{\mathbb C f}\mathcal D^\gg_e(M)$$ 
where $\Gamma_{\mathbb{C}f}$ is the functor of $\mathbb{C}f$-finite vectors.  
Set
$$\mathcal{E}M:=(\Gamma_{\mathbb C f}\mathcal D^\gg_e(M))/M,$$
cf.~\cite{E}.
Since $M\in \mathcal{C}_{\bar{\pp},\tt,n+2}$ satisfies the above assumptions, we have constructed a new functor
$$\mathcal E:\mathcal{C}_{\bar{\pp},\tt,n+2}\to\mathcal{C}_{\kk,n}.$$

\begin{lemma}\label{character} If $M\in \mathcal{C}_{\bar{\pp},\tt,n+2}$, $n\geq 0$, then for some $\gamma(\mu)\in\mathbb Z_{\geq 0}$
$$\Res_{\kk} M\simeq\bigoplus_{\mu\geq n+2}M_\kk(\mu)^{\oplus\gamma(\mu)}$$
and
$$\Res_{\kk}(\mathcal EM)\simeq\bigoplus_{\mu\geq n+2}V_\kk(\mu-2)^{\oplus\gamma(\mu)},$$
where $M_\kk(\mu):=U(\kk)\otimes_{U(\kk\cap\bar\pp)}\mathbb C_\mu$ for an integral $\tt$-weight $\mu$.
\end{lemma}
\begin{proof} The first assertion is obvious and the second follows from Lemma \ref{lem:localization} and the straightforward check that
$$(\Gamma_{\mathbb C f}\mathcal D^\kk_e(M_\kk(\mu))/M_\kk(\mu))=V_\kk(\mu-2).$$ 
\end{proof}

\begin{corollary}\label{cor:restriction}  If $M\in \mathcal{C}_{\bar{\pp},\tt,n+2}$, $n\geq 0$, then
$$\Res_\kk(\mathcal EM)\simeq\Res_\kk\left(\Gamma^1M\right).$$
\end{corollary}

\begin{corollary}\label{cor:fexact} The functor $\mathcal E:\mathcal{C}_{\bar{\pp},\tt,n+2}\to\mathcal{C}_{\kk,n}$ is exact.
\end{corollary}

\begin{prop}\label{prop:isomf} The functors $\mathcal E:\mathcal{C}_{\bar{\pp},\tt,n+2}\to\mathcal{C}_{\kk,n}$ and
$\Gamma^1:\mathcal{C}_{\bar{\pp},\tt,n+2}\to\mathcal{C}_{\kk,n}$ are isomorphic.
\end{prop}

\begin{proof} Let us start with the construction of a morphism of functors $\varphi:\mathcal E\to\Gamma^1$. Let $M\in \mathcal{C}_{\bar{\pp},\tt,n+2}$.
Then the exact sequence 
\begin{equation*}\label{sequence}
			0\to M\to \Gamma_{\mathbb Cf}\mathcal D^e_\gg(M) \overset{\pi_M}{\longrightarrow} \mathcal EM\to 0\,
		\end{equation*}
does not split over $\kk$, and therefore does not split over $\gg$.  Set 
$$ R^i(M):=\Gamma_\tt\Hom_{\mathbb C}\left(U(\gg)\otimes_{U(\tt)}\Lambda^i(\gg/\tt),M\right)$$  and let 
$$0\to M \xrightarrow{\partial_0} R^0(M)\xrightarrow{\partial_1} R^1(M)\xrightarrow{\partial_2} R^2(M)\xrightarrow{\partial_3}\dots$$
be the Koszul resolution of $M$ as introduced in Lemma 2.2 of \cite{Z}.

The complex $R^\bullet(M)$ is functorial with respect to $M$ and yields an injective resolution of $M$ in the category of $(\gg,\tt)$-modules.  Hence, we have a commutative diagram
\begin{equation*}
\begin{CD}
0@>>> M @>>>\Gamma_{\mathbb Cf}\mathcal D^e_\gg(M) @>>>\mathcal EM @>>>0\\
@VVV @V\operatorname{id}_MVV @V\eta_MVV @V\varphi_MVV@VVV\\
0@>>>M @>\partial_0>>R^0(M) @>\partial_1>>R^1(M)@>\partial_2>>R^2(M)
\end{CD}\notag
\end{equation*}
for some morphisms $\eta_M$ and $\varphi_M$, unique up to homotopy.  We recall from \cite{PZ5} that $\Gamma M=0$.  By construction, $\Gamma\left( \Gamma_{\mathbb Cf}\mathcal D^e_\gg(M)\right)=0$ and $\Gamma\mathcal{E}M=\mathcal{E}M$.  By applying $\Gamma$ to the above diagram, we obtain a new commutative diagram
\begin{equation*}
\begin{CD}
0@>>>0@>>>0@>>> \mathcal EM @>>>0\\
@VVV @VVV @VVV @V\Gamma\varphi_MVV @VVV\\
0@>>>0@>>> \Gamma R^0(M) @>\Gamma\partial_1>>\Gamma R^1(M) @>\Gamma\partial_2>>\Gamma R^2(M)\,.
\end{CD}\notag
\end{equation*}
The morphism $\Gamma\varphi_M$ induces a unique morphism $\psi_M:\mathcal{E}M\to\Gamma^1M$, by the definition of $\Gamma^1$.  Since our diagram is functorial in $M$, we obtain a morphism of functors $\psi:\mathcal{E}\to\Gamma^1$.

It remains to show that $\psi_M$ is an isomorphism for all $M\in\mathcal{C}_{\bar\pp,\tt,n+2}$.  Since both functors $\mathcal{E}$ and $\Gamma^1$ are exact, it is sufficient to check this for simple $M$.  Suppose that $\psi_M$ is nonzero.  Then, we have a surjective morphism $\psi_M:\mathcal{E}M\to\Gamma^1M$, since $\Gamma^1M$ is also simple.  But then, by Corollary~\ref{cor:restriction}, $\psi_M$ is an isomorphism.

On the other hand, assume $\psi_M=0$.  By a diagram chase, we obtain a nonzero morphism $\kappa_M:\mathcal{E}M\to R^0(M)$ such that $\partial_1 \kappa_M=\varphi_M$.  Moreover, $\im\kappa_M\simeq \im\eta_M$.  Because $\Gamma_{\mathbb Cf}\mathcal D^e_\gg(M)$ is an essential extension of $M$, $\eta_M$ is an injection, and hence $\Gamma\im\eta_M=\eta_M\Gamma\left( \Gamma_{\mathbb Cf}\mathcal D^e_\gg(M)\right)=0$.  On the other hand, 
$$\Gamma\im\kappa_M=\im\Gamma\kappa_M=\im\kappa_M\neq 0\,,$$
a contradiction.  Hence, $\psi_M\neq 0$.

%
\end{proof}

\begin{corollary}\label{cor:Verma} There is an isomorphism of $\gg$-modules $$Y(E)\simeq\mathcal E M(E).$$
\end{corollary}

We are now ready to give a proof of Lemma~\ref{lem:lem18}.

	\begin{proof}[Proof of Lemma~\ref{lem:lem18}] Set
$$D(E):=\mathcal D^\gg_e M(E),\quad C(E):=\Gamma_{\mathbb C  f}D(E)\,.$$
		From the explicit form of $D(E)$ as a $\kk$-module it is easy to check that 
		\begin{equation*}
			H_0\left(\nn_\kk,D(E)\right)=H_1\left(\nn_\kk,D(E)\right)=0\,.
		\end{equation*}
		By the spectral sequence of Proposition \ref{prop:spec} this implies 
		\begin{equation*}
			H_i\big(\nn,D(E)\big)=0\text{ for all }i\,.
		\end{equation*}
		The exact sequence
		\begin{equation*}
			0\to C(E)\to D(E)\to F(E)\to 0\,,
		\end{equation*}
		where $F(E):=D(E)/C(E)$, yields $H_2\big(\nn,C(E)\big)_{|F|}=H_3\big(\nn,F(E)\big)_{|F|}$.  It is easy to check that $H_0\left(\nn_\kk,F(E)\right)=0$, hence the input 
into $H_3\big(\nn,F(E)\big)$ in the spectral sequence (\ref{spectral}) comes from 
		\begin{equation}
			H_1\left(\nn_\kk,F(E)\right)\otimes\Lambda^2\left(\nn\cap\kk^\perp\right)\,.
			\label{eq:VHC}
		\end{equation}
		The maximum possible $\tt$-weight of $H_1\left(\nn_\kk,F(E)\right)$ is $2-|E|$, hence the maximum possible $\tt$-weight of (\ref{eq:VHC}) is $2-|E|+\lambda_1+\lambda_2$.  However, for any $F$ such that $X(F)\in \mathcal{C}_{\kk,n}$, we have $2-|E|+\lambda_1+\lambda_2<|F|$ as $|E|,|F|\geq \frac{\lambda_1+\lambda_2}{2}+2$.  We obtain $H_2\big(\nn,C(E)\big)_{|F|}=0$.  
		
		Next, we note that Corollary~\ref{cor:Verma} shows the existence of an exact sequence
		$$0\to M(E)\to C(E)\to Y(E)\to 0$$
		as $Y(E)\simeq \mathcal{E}M(E)$.  Since $M(E)$ is free as an $\nn$-module, $H_1\big(\nn,M(E)\big)=0$.  Together with $H_2\big(\nn,C(E)\big)_{|F|}=0$, this yields
		\begin{equation*}
			H_2\big(\nn,Y(E)\big)_{|F|}\simeq H_1\big(\nn,M(E)\big)_{|F|}=0
		\end{equation*}
		as $M(E)$ is free as an $\nn$-module.  The proof of Lemma~\ref{lem:lem18} is complete.
	\end{proof}

To prove Proposition \ref{prop:exact}, it now suffices to establish the following.	
	\begin{lemma}
		$\mathrm{B}_2X(E)=0$ for any $X(E)\in\mathcal{C}_{\kk,n}$.
		\label{lem:lem22}
	\end{lemma}
	
	\begin{proof}
		We will prove the statement by induction on the Bruhat height on the $\bb$-lowest weight of $L(E)$ ($\bb\subseteq \pp$ is the Borel subalgebra we fixed in Section~\ref{sec:prelim}).
		
		If $\lambda$ is $\bb$-dominant (i.e. $\bar{\bb}-$antidominant), $Y(E)=X(E)$ and we are done.  For an arbitrary $\lambda$, we consider the exact sequences
		\begin{equation}
			0\to N(E) \overset{\nu}{\longrightarrow} M(E)\overset{\mu}{\longrightarrow}L(E)\to 0
			\label{eq:ses1}
		\end{equation}
		and
		\begin{equation}
			0\to\Gamma^1 N(E) \to  Y(E)\to X(E)\to 0\,.
			\label{eq:ses2}
		\end{equation}
		The long exact sequence corresponding to (\ref{eq:ses2}) is the top row of the following commutative diagram
		\begin{equation*}
			\includegraphics[width=15.2cm]{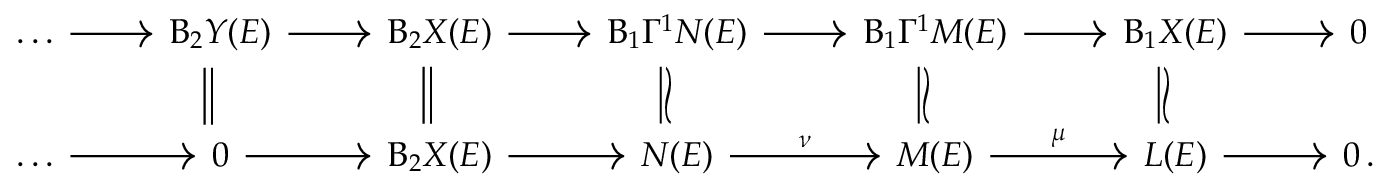}
		\end{equation*}
		The vertical isomorphisms are explained as follows:
		\begin{equation*}
			\mathrm{B}_2Y(E)=0\text{ by Lemma~\ref{lem:lem18}}\,,
		\end{equation*}
		\begin{equation*}
			\mathrm{B}_1\Gamma^1N(E)\simeq N(E)\text{ by the induction assumption}\,,
		\end{equation*}
		\begin{equation*}
			\mathrm{B}_1\Gamma^1M(E)\simeq M(E)\text{ by Corollary~\ref{cor:cor14}}\,,
		\end{equation*}
		and
		\begin{equation*}
			\mathrm{B}_1(X)\simeq L(E)\text{ by Proposition~\ref{prop:prop11}}\,.
		\end{equation*}
		The exactness of the bottom row of the diagram now yields $\mathrm{B}_2X(E)=0$, and Lemma~\ref{lem:lem22} is proved.  Proposition~\ref{prop:prop17} now follows.
	\end{proof}

\section{End of Proof of Theorem~\ref{main}}
\label{sec:endofproof}

	The results of Sections~\ref{sec:conjugacy}-\ref{sec:exactness} imply that, under the assumption $n\geq\Lambda$, the functors
	\begin{equation*}
		\Gamma^1:\mathcal{C}_{\bar{\pp},\tt,n+2}\rightsquigarrow\mathcal{C}_{\kk,n}
	\end{equation*}
	and
	\begin{equation*}
		\mathrm{B}_1:\mathcal{C}_{\kk,n}\rightsquigarrow\mathcal{C}_{\bar\pp,\tt,n+2}
	\end{equation*}
	are exact functors between finite-length abelian categories which induce mutually inverse bijections on isomorphism classes of simple objects. 
	
	The isomorphisms 
	\begin{equation*}
		\Hom_\gg\left(\mathrm{B}_1\Gamma^1 M,M\right)\simeq \Hom_\gg\left(\Gamma^1M,\Gamma^1M\right)
	\end{equation*}
	and
	\begin{equation*}
		\Hom_\gg\left(\mathrm{B}_1X,\mathrm{B}_1X\right)\simeq \Hom_\gg\left(X,\Gamma^1\mathrm{B}_1X\right)
	\end{equation*}
	for $X\in\mathcal{C}_{\kk,n}$ and $M\in \mathcal{C}_{\bar\pp,\tt,n+2}$ induce morphisms of functors 
	\begin{equation*}
		\Delta:B_1\circ \Gamma^1\rightsquigarrow \id_{\mathcal{C}_{\bar\pp,\tt,n+2}}
	\end{equation*}
	and
	\begin{equation*}
		\nabla:\Gamma^1\circ B_1\rightsquigarrow \id_{\mathcal{C}_{\kk,n}}\,.
	\end{equation*}
In addition, it follows from Section \ref{sec:bijection} that for any simple $M\in \mathcal{C}_{\bar\pp,\tt,n+2}$ the morphism 
$\Delta_M:B_1\Gamma^1 M\to M$ is an isomorphism.
	
We show now by induction on the length of objects that $\Delta_M$ is an isomorphism for any $M\in \mathcal{C}_{\bar\pp,\tt,n+2}$.  Consider the morphism $\Delta_M$, and let
	\begin{equation*}
		0 \to M_1 \to M \to M_2\to 0
	\end{equation*}
	be an exact sequence of modules in $\mathcal{C}_{\bar\pp,\tt,n+2}$ with $M_1\neq 0$ and $M_2\neq 0$.  Then, since both functors $\Gamma^1$ and 
$\mathrm{B}_1$ are exact, we have an exact sequence
	\begin{equation*}
		0\to \mathrm{B}_1\Gamma^1 M_1\to \mathrm{B}_1\Gamma^1M\to\mathrm{B}_1\Gamma^1M_2\to 0
	\end{equation*}
	in $\mathcal{C}_{\kk,n}$.  Since the lengths of $M_1$ and $M_2$ are less than the length of $M$, we can assume that we have a commutative diagram
	\begin{equation*}\label{diag}
	\includegraphics[width=10cm]{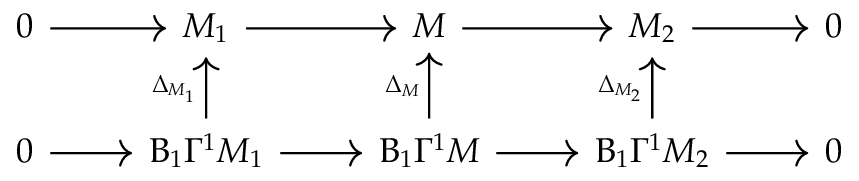}
	\end{equation*}
	where the left and the right vertical arrows are isomorphisms.  
By the short five lemma the middle vertical arrow $\Delta_{M}$ is an isomorphism too. 
	
	The case of $\nabla$ is analogous.

\section{Discussion and examples}
\label{sec:ex}
It is interesting to see when the functor $\mathrm{B_1}$ establishes an equivalence of the category $\mathcal{C}^\theta_{\kk,\Lambda}$ with the entire category
$\mathcal{C}^\theta_{\bar{\pp},\tt}$. This is equivalent to the question: for which central characters $\theta$  does the equality 
$\mathcal{C}_{\bar{\pp},\tt,\Lambda+2}^\theta= \mathcal{C}_{\bar{\pp},\tt}^\theta$ hold? 

Consider in more detail the case when $\gg$ is simple and $\kk$ is a principal subalgebra of $\gg$.
Here $h$ is a regular element of $\gg$ and $\pp=\bb$ is 
a Borel subalgebra. Let the simple roots of $\bb$ be $\alpha_1,\dots,\alpha_r\in\hh^*$, and $\beta$ be the highest root. 
Then $\beta=m_1\alpha_1+\dots+m_r\alpha_r$ for some positive integers $m_1,\dots,m_r$. 
Moreover, $\Lambda=\beta(h)-1$. We would like to find central characters $\theta$ such that 
$\mathcal{C}_{\bar{\bb},\tt,\Lambda+2}^\theta= \mathcal{C}_{\bar{\bb},\tt}^\theta$.
For a weight $\gamma\in\hh^*$ denote by $\theta_\gamma$ the central character of the Verma module $M(\gamma)=U(\gg)\otimes_{U(\bar\bb)}\mathbb C_\gamma$. The equality
$\theta_\gamma=\theta_\eta$ holds if and only if $\gamma-\rho$ and $\mu-\rho$ belong to the same orbit of the Weyl group, where $\rho$ is the half-sum of roots of $\bb$. 
This orbit contains a unique $\bb$-antidominant weight $\gamma-\rho$. Then $\gamma(h)<\eta(h)$ for any other $\eta=w(\gamma-\rho)+\rho$ on the Weyl group orbit. 
Therefore we need to find $\gamma$ such that $\gamma-\rho$ is antidominant and $\gamma(h)\geq \beta(h)+1$.

Let $h_1,\dots,h_r$ denote the simple coroots. Then $h=n_1h_1+\dots+n_rh_r$ for some positive integers $n_1,\dots,n_r$. We set
$\gamma_i:=\gamma(h_i)$. Since $\rho(h_i)=1$ for all $i=1,\dots,r$, the condition that $\gamma-\rho$ is antidominant can be written in the form
\begin{equation}\label{antidominant}
\gamma_i\leq 1\,\, \text{for all} \,\, i=1,\dots r.
\end{equation}
The equality $\alpha_i(h)=2$ shows that $\beta(h)=2\sum_{i=1}^r m_i$. Hence, the condition $\gamma(h)\geq\beta(h)+1$ is equivalent to
\begin{equation}\label{notrun}
\sum_{i=1}^r n_i\gamma_i\geq 1+2\sum_{i=1}^r m_i.
\end{equation}

Let $\Sigma(\gg)$ denote the set of weights satisfying conditions (\ref{antidominant}) and (\ref{notrun}). Clearly $\Sigma(\gg)$ is not empty
as soon as 
\begin{equation*}\label{notrun1}
\sum_{i=1}^r n_i\geq 1+2\sum_{i=1}^r m_i.
\end{equation*}
The latter inequality can be rewritten as
\begin{equation}\label{notrun1-1}
\rho(h)\geq 1+\beta(h).
\end{equation}

For example, let $\gg=\mathrm{sl}(r+1)$. Then $m_1=\dots=m_r=1$, hence $\beta(h)=1+2r$ and $\rho(h)=\frac{r(r+1)(r+2)}{6}$. Therefore (\ref{notrun1-1}) holds for $r\geq 3$.
\begin{prop} Let $\gg$ be a simple Lie algebra not equal to $\mathrm{sl}(2)$ or $\mathrm{sl}(3)$. Then $\Sigma(\gg)$ is not empty. If in addition $\gg\neq \mathrm{sp}(4)$, then 
$\Sigma(\gg)$ is infinite.
\end{prop}	
\begin{proof} $\Sigma(\gg)$ is infinite as soon as the inequality (\ref{notrun}) is strict. We can further rewrite (\ref{notrun1-1}) as
\begin{equation}\label{notrun2}
\frac{1}{2}\sum_{\alpha\in\Delta^+\setminus \beta}\alpha(h)\geq 1+\frac{1}{2}\beta(h).
\end{equation}
If $\beta(h)\geq 8$, the inequality \ref{notrun2} is strict as in this case the sum of positive non-highest $\tt$-weights in the $\kk$ submodule generated by the 
highest root vector is greater than the highest $\tt$-weight. Therefore, the statement holds for all $\gg$ of rank greater than $2$ and for $\gg=G_2$.
For $\gg=B_2$ we have $\rho(h)=7$, $\beta(h)=6$, and hence $\Sigma(\gg)$ consists of one element: $\Sigma(\gg)=\{\rho\}$.  For $\gg=A_2$ we have $\Sigma(\gg)=\emptyset$. 
\end{proof}

Note that the set of integral weights lying in $\Sigma(\gg)$ is always finite since $\Sigma(\gg)$ is compact. Moreover, the cardinality of this finite
set grows with rank.

Using translation functors we can strengthen Theorem \ref{main} for certain central characters. Let us call a central character $\theta$ $\kk${\it-adapted} if
$\mathcal{C}_{\bar{\bb},\tt,\Lambda+2}^\theta= \mathcal{C}_{\bar{\bb},\tt}^\theta$. A central character $\tilde\theta$ is {\it weakly} $\kk${\it-adapted} if there exists a 
$\kk$-adapted character $\theta$ and a translation functor $T$ establishing an equivalence between the categories of $\gg$-modules admitting respective generalized central 
characters
$\theta$ and $\tilde\theta$. Recall that, if $\tilde\theta=\theta_\eta$ for some $\eta$ such that $\eta-\rho$ is antidominant and $\theta=\theta_{\gamma}$ for some 
$\gamma\in \Sigma(\gg)$, then $\gamma-\eta$ must be integral and the stabilizers of $\gamma-\rho$ and $\eta-\rho$ in the Weyl group of $\gg$ must be the same \cite{BG}.
\begin{corollary} Assume that $\tilde\theta$ is weakly $\kk$-adapted. Then
\begin{enumerate}[(a)]
\item $\Gamma^1L$ is simple for any simple module $L\in \mathcal{C}^{\tilde\theta}_{\bar{\bb},\tt}$.

\item Let $\Gamma^1\mathcal{C}^{\tilde\theta}_{\bar{\bb},\tt}$ denote the full subcategory of 
$\mathcal{C}_{\kk,\Lambda}^{\tilde\theta}$ consisiting of modules whose simple constituents are of form $\Gamma^1L$ for simple modules $L\in \mathcal{C}^{\tilde\theta}_{\bar{\bb},\tt}$. 
Then the functor 
$\mathrm{B}_1:\Gamma^1\mathcal{C}^{\tilde\theta}_{\bar{\bb},\tt}\to\mathcal{C}^\theta_{\bar{\bb},\tt}$
is an equivalence of categories inverse to $\Gamma^1$.
\end{enumerate}
\end{corollary}
\begin{proof}

	Both assertions follow from the following commutative diagram of functors
	\begin{equation*}
	 \includegraphics[width=6cm]{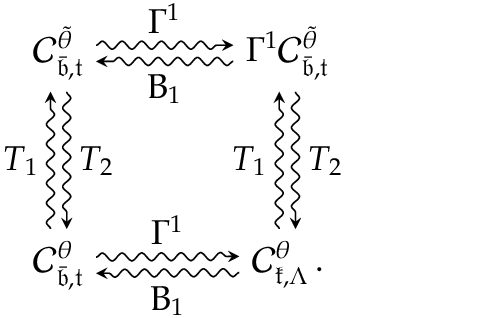}
	\end{equation*}
where $T_1,T_2$ are appropriate translation functors.  The commutativity of the diagram is a consequence of Theorem~\ref{main} and of the fact that the Zuckerman functor commutes with translation functors.  This latter fact is essentially a reformulation of Proposition 2.6 and Corollary 2.8 in~\cite{Z}.
\end{proof}

\section{An application}
\label{sec:appl} 
In this section $\kk$ is an arbitrary $\mathrm{sl}(2)$-subalgebra of $\gg$ and $n\geq \Lambda$.  By $\mathcal{C}^{\text{\rm ind}}_{{\kk},n}$ (respectively, $\mathcal{C}^{\text{\rm ind}}_{\bar\pp,\tt,n+2}$) we denote the category of inductive limits of objects from $\mathcal{C}_{\kk,n}$ (respectively, $\mathcal{C}^{\ell}_{\bar\pp,\tt,n+2}$).  Theorem \ref{main} implies the following

\begin{corollary} The functors $\Gamma^1$ and $\mathrm{B}_1$ induce mutually inverse equivalences of the categories $\mathcal{C}^{\text{\rm ind}}_{{\kk},n}$ and $\mathcal{C}^{\text{\rm ind}}_{\bar{\pp},\tt,n+2}$.
\end{corollary}

Recall that if an abelian category $\mathcal C$ has enough injectives, then the global dimension $\operatorname{gdim}\mathcal C$ is defined as
$$\operatorname{gdim}\mathcal C=\operatorname{sup}_{M,N\in\mathcal C}\left\{i\in\mathbb Z_{\geq 0}\, \big|\, \Ext_{\mathbb C}^i(M,N)\neq 0\right\}.$$

Corollary~\ref{cor:cor2} implies that $\mathcal{C}^{\text{ind}}_{\bar\pp,\tt,n+2}$ (and consequently also $\mathcal{C}^{\text{ind}}_{\kk,n}$ by Theorem~\ref{main}) has enough injectives.  The goal of this section is to prove the following proposition.
\begin{prop}\label{prop:gd} We have
$$\operatorname{gdim}\mathcal C^{\text{\rm ind}}_{{\kk},n}=\operatorname{gdim}\mathcal C^{\text{\rm ind}}_{\bar{\pp},\tt,n+2}\leq 2\dim \nn+\dim\cc-1$$
($\nn$ and $\cc$ are subalgebras of $\gg$ depending on the pair $\gg,\kk$ only).
\end{prop}

\begin{lemma}\label{lem:gd1} For every simple $C(\tt)$-module $E$ such that $|E|\geq n+2$, the module $\overline{W}(E)$ has an injective resolution in
$\mathcal C^{\text{\rm ind}}_{\bar{\pp},\tt,n+2}$ of length not 
greater than $\dim\nn$. Hence $\Ext_{\mathcal{C}_{\bar{\pp},\tt,n+2}}^i(M,\overline{W}(E))=0$ for any $M\in C^{\text{\rm ind}}_{\bar{\pp},\tt,n+2}$ and any $i>\dim \nn$.
\end{lemma}
\begin{proof} Recall the category  
$\FF_{\pp,\tt,n+2}$ from Section \ref{sec:prelim}. 
In this category $\bar E$ has an injective resolution  with terms
$$R^i_{\bar\pp}(E):=\left(\Gamma_\tt\Hom_{\mathbb C}(U(\pp)\otimes_{C(\tt)}\Lambda^i(\pp /C(\tt)), \bar E)\right)_{\geq n+2}.$$
Then $\Gamma_\tt\pro^\gg_\pp R^i_{\bar\pp}(E)$ provides an injective resolution for $\overline{W}(E)$ in $\mathcal C^{\text{\rm ind}}_{\bar{\pp},\tt,n+2}$ of length at most $\dim \nn$.
\end{proof}

\begin{lemma}\label{lem:gd2} For every simple $C(\tt)$-module $E$, let $W(E):=M(E)^\vee=\Gamma_\tt\pro^\gg_\pp(E)$. Then there exists an acyclic complex
$$0\to W(E)\to S^0\to\dots\to S^{\dim \cc-1}\to 0$$
such that all $S^i$ admit co-Verma filtrations.
\end{lemma}
\begin{proof} Let
$$T^i(E)=\Hom_{\mathbb C}\left(S^\bullet(\cc/\tt)\otimes\Lambda^i(\cc/\tt),E\right)\,.$$
Consider the exact complex of $C(\tt)$-modules
$0\to E\to T^0(E)\to T^1(E)\to\dots $
with usual Koszul differential and set
$S^i:=\Gamma_\tt\pro^\gg_\pp  T^i(E)$.
\end{proof}

\begin{corollary} \label{cor:gd1}  $\Ext_{\mathcal{C}_{\bar{\pp},\tt,n+2}}^i(M,W(E))=0$ for any $M\in \mathcal C^{\text{\rm ind}}_{\bar{\pp},\tt,n+2}$ and $i>\dim\nn+\dim\cc -1$. 
\end{corollary}
\begin{proof} We note that, by Lemma \ref{lem:gd1}, $\Ext_{\bar{\pp},\tt,n+2}^i(M,N)=0$ for $i>\dim \nn$ and $N$ admitting co-Verma filtration. In particular,
$\Ext_{\mathcal{C}_{\bar{\pp},\tt,n+2}}^i(M,S^j)=0$ for $i>\dim \nn$ and $j=0,\dots,\dim\cc-1$. Hence the statement. 
\end{proof}

\begin{lemma}\label{lem:gd3}  For every $E$, $L(E)$ has a right resolution of length not greater than $\dim \nn$ by modules which admit finite
filtrations with succesive quotients isomorphic to $W(F)$.
\end{lemma}
\begin{proof} This is a standard fact about parabolic category $\mathcal O$. Indeed let $R^i:=\Gamma_\tt\pro^\gg_\pp \Lambda^i(\gg/\pp)^*$.
Then $R^i\simeq S^\bullet((\gg/\pp)^*)\otimes\Lambda^i(\gg/\pp)^*$. Consider the Koszul complex
$$0\to R^0\to R^1\to\dots\to R^{\dim \nn}\to 0$$
as the complex of polynomial differential forms on the open orbit of $\bar P$ on $G/P$.  It gives a resolution of the trivial module by modules which admit
finite filtrations with succesive quotients isomorphic to $W(F)$. To obtain a similar resolution for $L(E)$, we tensor the above resolution with $L(E)$
and project on the subcategory of modules with the central character of $L(E)$.
\end{proof}

\begin{corollary} \label{cor:gd2}  $\Ext^i_{\mathcal{C}_{\bar{\pp},\tt,n+2}}(M,L(E))=0$ for any $M, L(E)\in \mathcal C^{\text{\rm ind}}_{\bar{\pp},\tt,n+2}$, and $i>2\dim\nn+\dim\cc -1$. 
\end{corollary}

Proposition \ref{prop:gd} follows from the last corollary since $\Ext^i_{\bar{\pp},\tt,n+2}(M,N)\neq 0$ implies $\Ext^i_{\mathcal{C}_{\bar{\pp},\tt,n+2}}(M,N')\neq 0$ for some submodule
$N'\subset N$ of finite length.


\begin{thebibliography}{20}
\bibitem [BBL] {BBL} G. Benkart, D. Britten, F. Lemire, {Modules with bounded weight multplicities for simple Lie algebras}, Math. Z. \textbf{225} (1997), 333--353.
\bibitem[BG]{BG} J. Bernstein, S. Gelfand,  Tensor products of finite
and infinite-dimensional representations
of semisimple Lie algebras, {Compositio Math.} \textbf{41}
(1980), 245--285.
\bibitem[BGG]{BGG} J. Bernstein, I. M. Gelfand, S. Gelfand, {A category of $\gg$-modules}, Funktional Anal. i Prilozhen {\bf5}, No. 2 (1971), 1--8; English translation, Functional Anal. and Appl. {\bf10} (1976), 87--92.
\bibitem [BL] {BL} D. J. Britten, F. W. Lemire, {A classification of simple Lie modules having a $1$-dimensional weight space}, Trans. Amer. Math. Soc. {\bf 299} (1987), 683--697.
\bibitem [D] {D} E. B. Dynkin, {Semisimple subalgebras of semisimple Lie algebras}, Mat. Sb. (N.S.) \textbf{30(72)}:2 (1952), 349--462.
\bibitem[E]{E} T. J. Enright, {On the fundamental series of a real semisimple Lie algebra}, Annals of Mathematics {\bf 110} (1979), 1--82.
\bibitem [EW] {EW} T. J. Enright, N. R. Wallach, {Notes on homological algebra and representations of Lie algebras}, Duke Math. J. \textbf{47} (1980), 1--15.
\bibitem [F] {F} V. Futorny, {The weight representations of semisimple finite dimensional Lie algebras}, Ph.D. thesis, Kiev University, 1987.
\bibitem [Fe] {Fe}  S. Fernando, {Lie algebra modules with finite dimensional weight spaces I}, Transactions of AMS \textbf{322} (1990), 757--781.
\bibitem [GMS] {GMS} I. M. Gelfand, R. A. Minlos, and Z. Ya. Shapiro., {Representations of the rotation and Lorentz groups and their applications}, translated from the Russian edition (Moscow, 1958) by G. Cummins and T. Boddington, Pergamon, London; Macmillan, New York, 1963.
\bibitem[GP]{GP} I. Gelfand, V. Ponomarev, {The category of Harish-Chandra modules over the Lie algebra of the Lorentz group}, Soviet. Math. Doklady {\bf 8}(5) (1967), 1065-1068.
\bibitem [GS1] {GS1} D. Grantcharov, V. Serganova, {Category of $\mathrm{sp}(2n)$-modules with bounded weight multiplicities}, Mosc. Math. J. {\bf 6} (2006), 119--134.
\bibitem [GS2] {GS2} D. Grantcharov, V. Serganova, {Cuspidal representations of $\mathrm{sl}(n+1)$}, Adv. Math. {\bf 224} (2010), 1517--1547.
\bibitem [HC] {HC} Harish-Chandra, {Infinite irreducible representations of the Lorentz group}, Proceedings of the Royal Society of London, Series A, Mathematical and Physical Sciences, vol. {\bf189}, Issue 1018 (1947), 372--401.
\bibitem [KV] {KV} A. W. Knapp, D. A. Vogan, Jr., {Cohomological induction and unitary representations}, Princeton University Press, Princeton, NJ, 1995.         
\bibitem [M] {M} O. Mathieu, {Classification of irreducible weight modules}, Annales de l'institut Fourier (Grenoble) \textbf{50} (2000), 537--592.
\bibitem[Pe]{Pe} A. Petukhov, {Bounded reductive subalgebras of $\mathrm{sl}_n$}, Transformation Groups \textbf{16}, 2011, 1173--1182.
\bibitem[PS]{PS1} I. Penkov, V. Serganova, {On bounded generalized Harish-Chandra modules}, Annales de l'Institut Fourier \textbf{62} (2012), 477--496.
\bibitem [PSZ] {PSZ} I. Penkov, V. Serganova, G. Zuckerman, {On the existence of $(\gg,\kk)$-modules of finite type}, Duke Math. J. \textbf{125} (2004), 329--349.
\bibitem [PZ1] {PZ1} I. Penkov, G. Zuckerman, {Generalized Harish-Chandra modules with generic minimal $\kk$-type}, Asian Journal of Mathematics \textbf{8} (2004), 795--812.
\bibitem [PZ2] {PZ2} I. Penkov, G. Zuckerman, {A construction of generalized Harish-Chandra modules with arbitrary minimal $\kk$-type}, Canadian Mathematical Bulletin \textbf{50} (2007), 603--609.
\bibitem [PZ3] {PZ5} I. Penkov, G. Zuckerman, {On the structure of the fundamental series of generalized Harish-Chandra modules}, Asian Journal of Mathematics \textbf{16} (2012), 489--514.
\bibitem [PZ4] {PZ6} I. Penkov, G. Zuckerman, {Algebraic methods in the theory of generalized Harish-Chandra modules}, Developments and Retrospectives in Lie Theory: Algebraic Methods, Developments in Mathematics, vol. \textbf{38}, Springer Verlag, 2014, 331--350.
\bibitem [V] {V} D. Vogan, {Representations of real reductive Lie groups}, Progress in Mathematics, vol. \textbf{15}, Birkhauser, Boston, 1981.
\bibitem [Z] {Z} G. Zuckerman, {Generalized Harish-Chandra modules}, in {Highlights of Lie algebraic methods}, Progress in Mathematics, vol. \textbf{295},
Birkhauser, 2012, 123--143.
\end{thebibliography}
\end{document}